\documentclass[final]{amsart}
\usepackage[utf8]{inputenc}
\usepackage{amsmath}
\usepackage{amssymb}
\usepackage{amsthm}
\usepackage{amscd}
\usepackage{mathtools}
\usepackage{marginnote}
\usepackage{todonotes}
\usepackage{bbm}
\usepackage{bm}
\usepackage[notcite,notref]{showkeys}
\usepackage[utf8]{inputenc}
\usepackage{tikz-cd}
\usepackage{wrapfig}
\usepackage{enumitem}
\usepackage{diagbox}
\usepackage{cite}
\usepackage[colorlinks=true,linkcolor=black,citecolor=black,urlcolor=blue]{hyperref} 
\newtheorem{theorem}{Theorem}[section]
\newtheorem*{theorem*}{Theorem}

\newtheorem*{conjecture*}{Conjecture}

\newtheorem*{question*}{Question}

\newtheorem*{guess*}{Guess}

\newtheorem*{problem*}{Problem}

\newtheorem{assumption}[theorem]{Assumption}
\newtheorem*{assumption*}{Assumption}
\newtheorem{lemma}[theorem]{Lemma}
\newtheorem*{lemma*}{Lemma}
\newtheorem{calculation}[theorem]{Calculation}
\newtheorem*{calculation*}{Calculation}

\newtheorem*{exercise*}{Exercise}
\newtheorem{proposition}[theorem]{Proposition}
\newtheorem*{proposition*}{Proposition}

\newtheorem{corollary}[theorem]{Corollary}
\newtheorem*{corollary*}{Corollary}

\theoremstyle{definition}

\newtheorem*{definition*}{Definition}
\newtheorem{remark}[theorem]{Remark}

\newtheorem*{example*}{Example}
\newtheorem*{examples*}{Examples}

\newcommand{\pres}[1]{\left\langle #1 \right\rangle}

\renewcommand{\bar}{\overline}
\usepackage{mathtools}

\renewcommand{\AA}{\mathbb{A}}

\newcommand{\FF}{\mathbb{F}}
\newcommand{\F}{\mathbb{F}}
\newcommand{\GG}{\mathbb{G}}

\newcommand{\PP}{\mathbb{P}}

\newcommand{\Q}{\mathbb{Q}}

\newcommand{\R}{\mathbb{R}}

\newcommand{\VV}{\mathbb{V}}

\newcommand{\ZZ}{\mathbb{Z}}
\newcommand{\Z}{\mathbb{Z}}

\newcommand{\Cc}{\mathcal{C}}

\newcommand{\Ec}{\mathcal{E}}
\newcommand{\Fc}{\mathcal{F}}

\newcommand{\Ic}{\mathcal{I}}

\newcommand{\Lc}{\mathcal{L}}
\renewcommand{\Mc}{\mathcal{M}}

\newcommand{\Oc}{\mathcal{O}}

\newcommand{\Vc}{\mathcal{V}}

\newcommand{\Xc}{\mathcal{X}}
\newcommand{\Yc}{\mathcal{Y}}
\newcommand{\Zc}{\mathcal{Z}}

\newcommand{\bfrak}{\mathfrak{b}}

\newcommand{\gf}{\mathfrak{g}}

\newcommand{\nf}{\mathfrak{n}}

\newcommand{\tf}{\mathfrak{t}}

\newcommand{\Xf}{\mathfrak{X}}

\newcommand{\onto}{\twoheadrightarrow}
\newcommand{\into}{\hookrightarrow}
\newcommand{\isomto}{\xrightarrow{\sim}}

\newcommand{\tensor}{\otimes}

\DeclareMathAlphabet{\mathpzc}{OT1}{pzc}{m}{it}

\newcommand{\red}{\operatorname{red}}

\newcommand{\psupp}{\operatorname{psupp}}

\newcommand{\Ext}{\operatorname{Ext}}
\DeclareMathOperator{\Homf}{\mathcal{H}\! \mathit{om}}

\newcommand{\Tor}{\operatorname{Tor}}

\newcommand{\ind}{\operatorname{ind}}
\newcommand{\Spec}{\operatorname{Spec}}

\newcommand{\tr}{\operatorname{tr}}

\newcommand{\Sym}{\operatorname{Sym}}
\newcommand{\ch}{\operatorname{char}}

\newcommand{\St}{\operatorname{St}}

\newcommand{\coker}{\operatorname{coker}}

\newcommand{\diag}{\operatorname{diag}}

\newcommand{\Cl}{\mathrm{Cl}}
\newcommand{\codim}{\mathrm{codim}}

\newcommand{\GL}{\mathrm{GL}}

\newcommand{\SL}{\mathrm{SL}}

\DeclareMathOperator{\Pic}{Pic}
\DeclareMathOperator{\Div}{Div}

\newcommand{\reg}{\mathrm{reg}}
\title{Singularities of Steinberg deformation rings}
\author{Daniel Funck \and Jack Shotton}
\date{\today}

\newcommand{\spl}{\mathfrak{sl}}

\begin{document}
\begin{abstract}
Let $l$ and $p$ be distinct primes, let $F$ be a local field with residue field of characteristic $p$, and let $\Xf$ be the irreducible component of the moduli space of Langlands parameters for $GL_3$ over $\Z_l$ corresponding to parameters of Steinberg type. We show that $\Xf$ is Cohen--Macaulay and compute explicit equations for it. We also compute the Weil divisor class group of the special fibre of $\Xf$, motivated by work of Manning for $GL_2$. Our methods involve the calculation of the cohomology of certain vector bundles on the flag variety, and build on work of Snowden, Vilonen--Xue, and Ngo. 
\end{abstract}
\maketitle

\section{Introduction}
\label{sec:intro}

Let $F$ be a local field with residue field $\F_q$ of order $q$, a power of $p$. Let $l$ be a prime distinct from $p$,
and let $E/\Q_l$ be a finite extension with ring of integers $\Oc$ and residue field $\F$. The (framed) moduli space of
 Langlands parameters for $GL_n$ over $\Oc$ parametrizes continuous homomorphisms \[\rho :
W^0_F \to GL_n(R)\] for $\Oc$-algebras $R$, where $W_F^0$ is a certain choice of discretization of the Weil group of
$F$. It has been intensely studied in recent years due to its role in modularity lifting theorems (for example,
\cite{Kisin2009-ModuliFFGSandModularity} and \cite{ClozelHarrisTaylor2008-Automorphy}) and, more recently, in
geometrization of the local Langlands conjecture as in \cite{Zhu20}, \cite{farguesGeometrizationLocalLanglands2021},
\cite{hellmannDerivedCategoryIwahoriHecke2020}, and \cite{ben-zviCoherentSpringerTheory2024}; the analogous object for
arbitrary reductive groups was constructed in \cite{DHKM20}, \cite{farguesGeometrizationLocalLanglands2021}, and
\cite{Zhu20}.

There is a very explicit presentation for the clopen subscheme of \emph{tame} parameters: it is isomorphic to
the $\Oc$-scheme $\Mc = \Mc_{n,q}$ whose $R$-points, for $\Oc$-algebras $R$, are given by 
\[\Mc_{n,q}(R) = \{(\Phi, \Sigma) \in GL_n(R) \times GL_n(R) : \Phi\Sigma\Phi^{-1} = \Sigma^q\}.\] This has a number of
pleasant geometric properties: it is an affine complete intersection over $\Oc$, equidimensional of relative dimension
$n^2$, it is flat over $\Oc$, and it is reduced (see \cite{DHKM20} section 2).

The geometric irreducible components of $\Mc$ are in bijection with $q$-power stable conjugacy classes in $GL_n(\bar{E})$.
In applications of the Taylor--Wiles--Kisin method, as well as in the analysis of the functors conjectured in
\cite{Zhu20} and \cite{farguesGeometrizationLocalLanglands2021}, it is natural to consider various
unions of irreducible components of $\Mc$ obtained by restricting the conjugacy class of $\Sigma$ (the `type'). For $n =
2$, very explicit equations for these fixed-type deformation rings were constructed by the second author in \cite{Shotton16,hupaskunas2019}. As a result, they were shown to be Cohen--Macaulay. 

In this article we consider the \emph{Steinberg component} $\Xc_{\St}$ defined to be the Zariski closure of the
open subset of $\Mc(\bar{E})$ on which $\Sigma$ is a regular unipotent matrix: 
\[\Xc_{St} = \overline{\{(\Phi, \Sigma) \in \Mc(\bar{E}) : \Sigma \text{ regular unipotent}\}}.\]
Our main theorem is then:
\begin{theorem} \label{thm:main} Let $n = 2$ or $3$ and suppose that $l > n$ and that \[q \equiv 1 \mod l.\]

Then $\Xc_{\St}$ is Cohen-Macaulay and $\Xc_{\St, \F}$ is normal and reduced.
\end{theorem}

Under the same hypotheses, we also obtain a complete set of equations for $\Xc_{\St}$ as a closed subscheme of $GL_n \times GL_n$; \emph{a priori}, it is defined only as a Zariski closure. See Corollaries~\ref{cor:Xc-eqns-sl2} and~\ref{cor:Xc-equations-sl3}.

When $n = 2$ this theorem (except the very last part) was essentially proved by the second author in \cite{Shotton16} by explicit
methods, but we reprove it here more geometrically. The calculations of \cite{Shotton16} were applied by Manning in \cite{manning} to compute the
Weil divisor class group of $\Xc_{\St, \F}$. This was combined with the Taylor--Wiles--Kisin patching method and a
certain self-duality argument to identify the `patched module' as an element of the class group and deduce multiplicity $2^k$ results for the mod $l$ cohomology of Shimura curves (and sets) associated to division
algebras ramified at $p$.

For $GL_3$ we prove
\begin{theorem}\label{thm:class-group} The Weil divisor class group of $\Xc_{\St, \F}$ is isomorphic to $\Z \times
\Z/3\Z$.
\end{theorem}
Our methods also give a new method to calculate the Weil divisor class group when $n = 2$, which Manning did using toric
geometry; for $n = 3$ the variety is no longer toric. 

Unfortunately, and unlike for $n = 2$, the expected self-duality property does not specify a unique element of the class group (rather, there are three possibilities). We therefore do not make a precise conjecture for what the patched module should be, nor a multiplicity conjecture. We also note that there are a number of other obstructions to carrying out Manning's argument in the case $n = 3$; see section~\ref{sec:multiplicity-new} for further discussion.

The Steinberg component was previously considered by the first author \cite{funck2023geometry}, in which he showed
that $\Xc_{\St}$ is smooth when $l$ is a banal prime, meaning that $q^i \not\equiv 1 \bmod l$ for $1 \le i \le n$. In fact he proved an
analogous result for general $G$ when $q$ is ``considerate" towards $l$, a notion related to (but not identical with)
the generalisations of ``banal" considered in \cite{DHKM20}. This paper in some sense deals
with the other extreme, when $q \equiv 1 \bmod l$, sometimes called the quasi-banal case (if $l > n$).

We use the method developed by Snowden \cite{snowdenSingularitiesOrdinaryDeformation2018} to study ordinary deformation
rings. The idea is to consider a projective resolution of $\Xc_{\St}$ by the variety obtained by adding a choice of
Borel subgroup containing $\Phi$ and $\Sigma$. This variety fibres over the flag variety for $GL_3$, and the desired
properties of $\Xc_{\St}$ can then be proved by computing (enough of) the cohomology of certain vector bundles on the
flag variety. The methods for computing these cohomology groups are mostly due to Vilonen and Xue \cite{VilonenXue} and Ngo
\cite{Ngo18}, who were motivated by generalising Snowden's work in \cite{snowdenSingularitiesOrdinaryDeformation2018} on ordinary deformation rings; however, we have to go  beyond their calculations due to the subtleties of working in positive characteristic and also to obtain our results on explicit equations and multiplicities. The idea of applying this method in the $l \ne p$ setting is also original to this paper. 

Similar resolutions play a role in the ``generalised Springer theory" of \cite{Zhu20} and
\cite{ben-zviCoherentSpringerTheory2024} and it would be interesting to compare our results with Conjecture~4.5.1 of
\cite{Zhu20}; however, we will not do this here.

Here is an outline of the paper. In section~\ref{sec:resolution} we define the resolution we need and explain how
Theorem~\ref{thm:main} follows from properties that hold for its special fibre. In section~\ref{sec:vec-buns} we prove
these properties modulo calculations of the cohomology of vector bundles on the flag variety; we carry out these calculations in section~\ref{sec:cohom-calc}. In section~\ref{sec:equations} we obtain explicit equations for $\Xc_{\St}$. In
section~\ref{sec:class-group} we use our resolution to compute the Weil divisor class group of $\Xc_{\St, \F}$ and identify the
canonical class inside it. We identify the divisorial sheaves satisfying the expected self-duality, and compute their multiplicities.
\subsection{Acknowledgements}
\label{sec:ack}
We thank Jeffrey Manning for helpful correspondence about this research.

Much of this work was carried our during the first author's PhD, funded by the Engineering and Physical Sciences Research Council.
\section{Resolution of Steinberg deformation rings}
\label{sec:resolution}

We let $p$, $q$, $l$, $\Oc$, $E$ and $\F$ be as in the introduction, let $n \ge 1$ be an integer, and consider the
Steinberg component $\Xc_{\St}$ as defined in the introduction; equivalently, $\Xc_{\St}$ is the Zariski closure of
\[\{(\Phi, \Sigma) \in GL_n(\bar{E}) \times GL_n(\bar{E}) : \Phi \Sigma \Phi^{-1} =
\Sigma^q, \Sigma \text{ regular unipotent}\}.\]
inside $GL_{n,\Oc} \times GL_{n,\Oc}$. 

From now on, we assume:
\begin{assumption} \label{ass:nlq} We have $l > n$ and $q \equiv 1 \bmod l$.
\end{assumption}

On $\Xc_{St}$, the eigenvalues of $\Phi$ are in the ratio $1:q:\ldots:q^{n-1}$. Since $q \equiv 1 \mod l$ and $l \nmid
n$, there is then an isomorphism $\GG_m \times \Xc \isomto \Xc_{\St}$ where
\begin{align*}\Xc &= \{(\Phi, \Sigma) \in \Xc_{\St} : \tr(\Phi) = 1 + q + \ldots + q^{n-1}.\} \\
    &= \{(\Phi, \Sigma \in \Xc_{\St} : \ch_{\Phi}(T) = \prod_{i=1}^{n-1}(T-q^i))\}.
\end{align*}
It will be technically more convenient to work with $\Xc$. Here $\ch_{\Phi}(T)$ is the characteristic polynomial of $\Phi$.

As $l > n$, the logarithm map $\Sigma\mapsto \log(\Sigma-1)$ is well defined for unipotent $\Sigma$ and hence on $\Xc$
we may write $\Sigma = \exp(N)$ for a nilpotent matrix $N \in \gf$, where $\gf$ is the Lie algebra of $GL_n$; the
defining equation then becomes $\Phi N \Phi^{-1} = qN$. In what follows we will describe points on $\Xc$ in the form
$(\Phi, N)$.

We let $X = \Xc_{\F}^{\red}$; it will follow from our work below that
\[X = \overline{\{(\Phi, \Sigma) \in \Xc(\bar{\FF}) : \Sigma \text{ regular nilpotent}\}},\] which is irreducible. This
fact also follows from the results of \cite{shotton-gln} Section~7.

\subsection{Resolution of $\Xc$}
\label{sec:res-def}

Let $B \subset GL_n$ be the standard Borel subgroup with Lie algebra $\bfrak$; let $N$ be its unipotent radical, with
Lie algebra $\nf$. Let $\Fc \cong G/B$ be the flag variety; for an $\Oc$-algebra $R$, we can write a point $F \in \Fc(R)$ as a flag
$0 \subset F_{n-1} \subset \ldots \subset F_0 = R^n$ with the $F_i$ projective $R$-modules such that
$\mathrm{gr}_i(F_{\bullet})$ are all projective. We define $\Zc$ to be the closed subscheme of $G \times \gf \times \Fc$
given (on $R$-points) as the set of triples $(\Phi, N, F)$ such that 

\begin{align*}(\Phi - q^i)F_i &\subset F_{i+1}\\ \intertext{and}
    NF_i&\subset F_{i+1} 
\end{align*}
for $i = 0, \ldots, n-1$. We define $\Yc \subset \Zc$ as the closed subscheme with
\[\Yc(R) = \{(\Phi, N, F) \in \Zc(R) : \Phi N \Phi^{-1}= qN\}.\]

We thus have a closed embedding $\Yc\hookrightarrow \Zc$ fitting into the diagram below:

\[\begin{tikzcd}
	\Zc && \Yc && \Xc \\
	\\
	& \Fc
	\arrow["\pi_{\Zc}"', from=1-1, to=3-2]
	\arrow["{\pi}", from=1-3, to=3-2]
	\arrow[hook', from=1-3, to=1-1]
	\arrow["f", from=1-3, to=1-5]
\end{tikzcd}\]

\begin{lemma}\label{lem:isom-regular} The morphism $f : \Yc \to \Xc$ given by forgetting $F$ is a projective morphism. It is an isomorphism over the open subset of $\Xc$ on which $N$ is regular or $l$ is invertible.
\end{lemma}
\begin{proof}
The scheme $\Yc$ is a closed subscheme of $\Xc\times \Fc$ and $\Fc$ is projective, thus $\Yc \to \Xc$ is projective. Let
$U$ denote the open subset of $\Xc$ on which $N$ is regular or $l$ is invertible. We write down an inverse to $f$ on $U$
by writing down the required flag for each $R$-point $(\Phi, N)$ of $U$:
\begin{itemize}
    \item When $N$ is regular (that is, its value at each point of $\Spec R$ is regular), take $F_i = \ker(N^{n-i})$;
    \item When $l$ is invertible, take $F_i=\bigoplus_{j=i}^{n-1} \ker(\Phi-q^i)$.
\end{itemize}
The relation $\Phi N \Phi^{-1} = qN$ implies that these agree on the locus where  $N$
is regular \emph{and} $l$ is invertible, and one can check that this defines a two-sided inverse $U \to f^{-1}(U) \subset \Yc$ of $f|_{f^{-1}(U)}$.
\end{proof}

\begin{lemma}\label{lem:ZYproperties}
\begin{enumerate}
    \item The scheme $\mathcal{Z}$ is an affine bundle over $\Fc$; in particular, it is $\Oc$-flat and
    $\mathcal{Z}_\F$ is reduced and irreducible.
    \item If $n \le 3$, $\mathcal{Y}$ is $\Oc$-flat and $\mathcal{Y}_{\FF}$ is reduced and irreducible. It is a local
    complete intersection, and hence Cohen--Macaulay.
\end{enumerate}
\end{lemma}

\begin{proof}
\begin{enumerate}
    \item We may cover $\Fc$ by open affine subschemes $U$, with $U=\Spec(A)$, such that the projection $\GL_n \rightarrow
    \Fc$ has a section $\gamma:U\rightarrow \GL_n$. Notice that $\gamma \in GL_n(A)$, so the universal pair $(\Phi,N)$
    on $\pi_{\Zc}^{-1}(U)$ takes the form 
    \[(\gamma(\Phi_0+M)\gamma^{-1},\gamma N \gamma^{-1})\] with $\Phi_0=\textrm{diag}(q^{n-1},...,q,1)$ and $M,N\in
    \nf$. It is now easy to see that $\pi_{\Zc}^{-1}(U)\cong U\times \nf^2$; examining the behaviour of $M$ under
    conjugation by $B$ gives that $\Zc$ is an affine bundle.

    \item When $n=2$, $\Yc=\Zc$, so we are done. 
    
    For $n=3$, the argument of part 1 gives similarly that $\Yc \times_{\Fc}U\cong U\times \Cc(\nf)$, where 
    \[\Cc(\nf)=\{(M,N)\in \nf\times \nf| (\Phi_0+M)N-qN(\Phi_0+M)=0\}\]
    Let \[M=\begin{pmatrix}
        0 & a & b \\
        0 & 0 & c \\
        0 & 0 & 0
    \end{pmatrix}\] and 
    \[N=\begin{pmatrix}
        0 & d & e \\
        0 & 0 & f \\
        0 & 0 & 0 \end{pmatrix}.\] Then the equation defining $\Cc(\nf)$ is $(q^2-1)e+af-dc=0$. This equation is not
    divisible by a uniformiser of $\Oc$, so $\Cc(\nf)$ and hence $\Yc$ is $\Oc$-flat. Now, $\Cc(\nf)_\FF$ is the affine
    hypersurface in $\Spec \FF[a,b,c,d,e,f]$ defined by $af - cd = 0$. Since $af-cd$ is irreducible, the rest of 
    part~(2) follows.\qedhere
\end{enumerate}
\end{proof} 

Recall that $X = \Xc_\F^{\red}$, and define $Y = \Yc_{\FF}$. Since $Y$ is reduced, the morphism $Y \to \Xc_\F$ factors
through $X$. Abusing notation slightly, we also write $f$ for this map $Y \to X$. 
The proof of the next theorem occupies most of the next two sections.
\begin{theorem}\label{thm:resolution} Suppose that $n \le 3$ (and recall that Assumption~\ref{ass:nlq} entails $l > n$).
\begin{enumerate}
\item The morphism $f : Y \to X$ is birational, \[R^if_*\Oc_Y = R^if_*\omega_Y = 0\] for $i > 0$, and the induced map $\Oc_X
\to f_*\Oc_Y$ is an isomorphism.
\item The variety $X$ is Cohen-Macaulay, with $X  =\Spec \Gamma(Y, \Oc_Y)$, and $\omega_X \cong f_*\omega_Y$.
\end{enumerate}
\end{theorem}

\begin{proof}Part~(2) of Theorem~\ref{thm:resolution} follows from part~(1) by Lemma 2.1.4 of
\cite{snowdenSingularitiesOrdinaryDeformation2018} and the fact that $Y$ is Cohen--Macaulay. We will prove part~(1) in
the next section.
\end{proof}

\begin{remark}
    In fact $Y$ has resolution-rational singularities --- because of the corresponding fact for the cone on a smooth
    quadric in $\PP^3$ --- and it follows that $X$ also has resolution-rational singularities, the strongest of various
    possible notions of rational singularity in positive characteristic discussed in \cite{KovacsRS}. 
\end{remark}

\subsection{Proof of Theorem \ref{thm:main}}

In this section, we explain how to deduce Theorem~\ref{thm:main} from Theorem \ref{thm:resolution}. As
$\Xc_{St}\cong\Xc\times \GG_m$, we need only prove the theorem for $\Xc$. Let $B = \Gamma(\Yc, \Oc_{\Yc})$ and $A =
\Gamma(\Xc, \Oc_{\Xc})$. Thus we have a morphism $f^*:A \to B$ that we would like to show is an isomorphism.

By flat base change, $B \otimes_{\Oc} E = \Gamma(\Yc_E, \Oc_{\Yc_E}) = \Gamma(\Xc_E, \Oc_{\Xc_E}) = A \otimes_{\Oc}E$.
Moreover, as $f:\Yc \to \Xc$ is proper, $B$ is a finite $A$-algebra.

\begin{lemma}
We have $B \otimes_{\Oc}\FF = \Gamma(Y, \Oc_{Y})$.
\end{lemma}
\begin{proof}
Using the short exact sequence $0 \to \Oc_{\Yc} \xrightarrow{\times \varpi} \Oc_{\Yc} \to \Oc_{Y} \to 0$ it suffices to
show that $H^1(\Yc, \Oc_{\Yc}) = 0$. Using that $\Xc$ is affine, this is equivalent to showing that $R^1f_{*}(\Oc_{\Yc})
= 0$. By Theorem~\ref{thm:resolution}, $R^if_{*}(\Oc_{Y}) = 0$ for $i \ge 1$ and, in particular, $(R^1f_*
\Oc_{\Yc})\otimes \FF = 0$ (using the short exact sequence again). But we also have $R^1f_*\Oc_{\Yc} \otimes E = 0$ by
flat base change and the fact that $f$ is an isomorphism after inverting $l$. Now $R^1f_*\Oc_{\Yc}$ is a
finitely-generated $A$-module $M$ (as $f$ is proper)  such that $M \otimes_\Oc \FF = M \otimes_\Oc E = 0$, from which it
follows that $M = 0$. 
% But any such module must be zero: since $\Spec A =\Spec(A \otimes \FF)\cup \Spec(A \otimes E)$, all fibres of $M$ are
% zero, so $M$ is zero by Nakayama's lemma.
\end{proof}

\begin{proposition} \label{prop:AeqB}
The map $A \to B$ is an isomorphism.
\end{proposition}
\begin{proof}
We know the proposition after inverting $l$. We claim that $A \to B$ is surjective. After $\otimes \FF$, this follows as, by the previous lemma and Theorem~\ref{thm:main} (2), $B \otimes_{\Oc} \FF = \Gamma(Y, \Oc_Y) = (A\otimes_{\Oc} \FF)^{\red}$. But then the cokernel, a finite $A$-module, vanishes after $\otimes E$ and $\otimes \FF$ and so must be zero, as at the end of the previous proof. 

Now, $A \to B$ is a surjective map of flat $\Oc$-algebras that is an isomorphism after inverting $l$, and is therefore an isomorphism, as required.
\end{proof}

\begin{proof}[Proof of Theorem~\ref{thm:main}]
   Since $A$ is $\varpi$-torsion free, $\varpi$ is a regular element of $A$. By Proposition~\ref{prop:AeqB}, $A
   \otimes_\Oc \FF = B \otimes_\Oc \FF = \Gamma(Y, \Oc_Y)$. This algebra is reduced since $Y$ is reduced, and so $X = \Spec(A \otimes_\Oc \FF)$.
   By Theorem~\ref{thm:resolution}, $X$ is Cohen--Macaulay. Finally,
   we will show below in Lemma~\ref{lem:YtoF} that the singular locus of $X$ has codimension 2, and so $X$ is normal and reduced by Serre's criterion and the fact that it is Cohen--Macaulay.
\end{proof}

\section{Vector bundles on the flag variety}
\label{sec:vec-buns}

Our aim in this section is to prove Theorem~\ref{thm:resolution}, modulo technical cohomological calculations that we
defer until later.

Recall that $X = \Xc_\F^{\red}$ and $Y = \Yc_\F$, and similarly define $Z = \Zc_\F$ and $F = \Fc_\F$. Note that $Y
\subset Z$ is a closed subscheme (and this is an equality for $n=2$). We continue to write $\pi$ and $\pi_Z$ for the natural morphisms $Y \to F$ and $Z \to F$, respectively. This gives us the following diagram of varieties over $\F$:

\[\begin{tikzcd}
	Z && Y && X \\
	\\
	& F
	\arrow["\pi_Z"', from=1-1, to=3-2]
	\arrow["{\pi}", from=1-3, to=3-2]
	\arrow[hook', from=1-3, to=1-1]
	\arrow["f", from=1-3, to=1-5]
\end{tikzcd}\]

Since $\pi_Z$ is affine, for any coherent sheaf $\Vc$ on $Z$ we have $H^i(Z, \Vc) = H^i(F, \pi_{Z,*}\Vc)$ (and similarly for
sheaves on $Y$). This is the starting point of our analysis.

\subsection{Bundles, roots, weights}
\label{sec:buns}
Working always over the field $\F$, we let $G = SL_n$ (a small departure from the previous section) and take $B$ to be
the standard Borel subgroup of upper triangular matrices, with $T$ the standard torus and $U$ the unipotent radical. Let
$\gf$, $\bfrak$, $\tf$ and $\nf$ be their respective Lie algebras. We write $X(T)$ for the character group of $T$, choose a system of positive 
roots such that the weights of $\nf$ are \emph{negative}, and write $\rho$ for half the sum of the positive roots.

If $V$ is a variety over $\F$ with an action of $B$, then we can form $G \times^B V$, which is a fibre bundle over $F$
with fibre $V$, and carries an action of $G$ compatible with that on $G/B$. 

If $V$ is the vector-space scheme underlying a finite-dimensional representation of $B$, then we obtain a
$G$-equivariant vector bundle on $F$, and this furnishes an equivalence of abelian categories between finite-dimensional
representations of $B$ over $\F$ and $G$-equivariant locally free coherent sheaves on $F$. We will therefore identify finite-dimensional representations of $B$ with the corresponding $G$-equivariant coherent sheaves.

The representation $\gf$ is self-dual via the trace pairing and under
this pairing $\bfrak = \nf^{\perp}$; thus $\nf^* \cong \gf/\bfrak$ as $B$-representations. If $\chi \in X(T)$ is a character, we write $\Oc(\chi)$ for the corresponding $G$-equivariant
line bundle on $F$. For $\Ec$ a coherent sheaf on $F$ we write $\Ec(\chi) = \Ec \otimes_{\Oc_F} \Oc(\chi)$. 

\subsection{The varieties $Y$ and $Z$.}
\label{sec:YZ}
Examining the proof of Lemma~\ref{lem:ZYproperties} and noting that $\Phi_0$ there is the \emph{identity} matrix over
$\F$, we see that  
\[Z \cong G \times^B (\nf \times \nf)\]
and that 
\[Y \cong G \times^B C(\nf)\]
where
\[C(\nf) = \{(M, N) \in \nf \times \nf : [M,N] = 0\}.\] In particular,
$Z$ is a vector bundle over $F$. 

\begin{lemma}\label{lem:birational} Suppose that $n = 2$ or $n = 3$. Then the morphism $f: Y \to X$ is birational.
\end{lemma}
\begin{proof}
    It follows from the above that $Y$ is irreducible. Since $\Yc \to \Xc$ is proper, surjective over $E$, and $\Xc$ is
    flat over $\Oc$, we see that $f$ is surjective. As $n \le 3$, $Y$ is irreducible, and so the same is
    true for $X$. Moreover, $\dim Y = n^2 - 1$ (again, for $n \le 3$), and this is also the dimension of the set of
    points $(M,N,F)\in Y$ such that $N$ is regular. But $f$ is an isomorphism on this locus, and is therefore
    birational.
\end{proof}

If $\Ec$ is a coherent sheaf on $Z$, we write $\Ec(\chi)=\Ec\otimes
_{\Oc_Z}\pi_Z^*\Oc(\chi)$. The projection formula then gives the compatibility
\[\pi_{Z,*}\Ec(\chi) \cong (\pi_{Z,*}\Ec)(\chi).\]
Similarly for coherent sheaves on $Y$.

\begin{proposition}{\label{prop:sheafcalcs}}  We have the following isomorphisms:
\begin{align}
     \pi_{Z,*} \Oc_Z &\cong \Sym[(\gf/\bfrak)^2]\\
     \omega_Z &\cong \Oc_Z(2\rho)\\
\intertext{and, if $n = 3$, then}
     \Ic_Y &\cong\Oc_Z(\rho)\\
     \omega_Y &\cong \Oc_Y(\rho).
\end{align}
\end{proposition}
\begin{remark}
    Part~(1) is implicit in \cite{VilonenXue}, while part~(3) may be found in \cite{Ngo18}. We include a proof here
    since we also require the result in positive characteristic, as well as the versions with the dualizing sheaf.
\end{remark}
\begin{proof}
\begin{enumerate}
    \item We have that $Z$ is the total space of the $G$-equivariant vector bundle $\nf^2$, and $\nf^* \cong \gf/\bfrak$. Therefore 
    \[Z=\underline{\Spec}_{F}(\Sym[(\gf/\bfrak)^2]),\]
     and the result follows.
    \item Recall that $\omega_Z\cong\omega_{Z/F}\otimes_{\Oc_Z} \pi_Z^*\omega_F$ because $\pi_Z$ and $F$ are smooth. We also recall from \cite{Jantzen03} II. \S 4.2 that $\omega_F\cong \Oc(-2\rho)$. It remains to calculate
    $\omega_{Z/F}$. Notice that the relative tangent bundle has $T_{Z/F}=\pi_Z^*((\nf)^2)$, so $\omega_{Z/F}\cong
    \pi_Z^*\det((\gf/\bfrak)^2)=\pi_Z^*(\Oc(4\rho))$. Therefore
    \[\omega_Z\cong \pi_Z^*(\Oc(-2\rho)\otimes\Oc(4\rho))=\pi_Z^*(\Oc(2\rho)) = \Oc_Z(\rho).\]
    \item Write a general point of $\nf^2$ as 
    \[(M, N) = \left(\begin{pmatrix}
        0 & a & b \\ 0 & 0 & c \\ 0 & 0 & 0
    \end{pmatrix},\begin{pmatrix}
        0 & d & e \\ 0 & 0 & f \\ 0 & 0 & 0
    \end{pmatrix}\right)\]
    as in Lemma~\ref{lem:ZYproperties}. Then $a, b, \ldots, f$ are a basis for $(\nf^2)^*$. The ideal sheaf of $C(\nf)$ inside $\Sym[(\nf^2)^*]$ is principal, generated by $af - dc$, which is an element of $\Sym^2[(\nf^2)^*]$ of weight $\rho$.\footnote{Recall that our convention is that the weights of $\nf$ are \emph{negative}.} In other words, \[\Ic_{C(\nf)} \cong \Oc_{\nf \times \nf}(\rho)\] as $B$-equivariant sheaves on $\nf^2$. 
    Since $Y = G\times^B C(\nf)\subset G \times^B\nf^2 = Z$ we get that 
    \[\Ic_Y \cong \Oc_Z(\rho)\]
    as required.
    \item This follows from the adjunction formula (see \cite[\href{https://stacks.math.columbia.edu/tag/0AU3}{Section 0AU3}~(7)]{stacks-project})
    \[\omega_Y = i^*(\omega_Z\otimes_{\Oc_Z} \Ic_Y^\vee)\]
    and parts~(2) and~(3).\qedhere
\end{enumerate}
\end{proof}

\subsection{Proof of Theorem~\ref{thm:resolution}.} 

We first prove Theorem~\ref{thm:resolution} when $n = 2$. This essentially appears (in the context of ordinary
deformation rings when $l = p$) in section 3.3 of \cite{snowdenSingularitiesOrdinaryDeformation2018}, but we include the
proof here to illustrate the ideas in this setting and to prepare the ground for the more complicated case $n = 3$.  

\begin{proof}[Proof of Theorem~\ref{thm:resolution} when $n=2$.]
In this case $F=\PP^1$ and $Y=Z$ is the total space of the vector bundle $\nf^2$. As line
bundles on $\PP^1$, we see that \[\gf/\bfrak \cong \nf^*\cong \Oc(2)\] and $\Oc(\rho)\cong \Oc(1)$, so that
\[H^1(\PP^1,\Sym^r\left[(\gf/\bfrak)^2\right])=H^1(\PP^1,\Sym^r\left[(\gf/\bfrak)^2\right](2\rho))=0\] 
for every $r\geq0$. By Proposition~\ref{prop:sheafcalcs},  
\[H^1(Y,\Oc_Y)=H^1(F,\pi_*\Oc_Y)=0\] and
\[H^1(Y,\omega_Y)=H^1(F,\pi_*\omega_Y)=0.\]
Since $X$ is affine, $H^1(Y,\Oc_Y) = H^0(X, R^1f_*\Oc_Y)$ vanishes and so \[R^1f_*\Oc_Y = 0.\]
Similarly, $R^1f_*\omega_Y= 0$.

It remains to show that $f^*:H^0(X,\Oc_X)\hookrightarrow H^0(Y,\Oc_Y)$ is an isomorphism.
For convenience, we set $R=H^0(X,\Oc_X)$ and $\tilde{R}=H^0(Y,\Oc_Y)$.
We note that the natural map
\[H^0(F,\Sym\left[\gf^2\right]) = \F[\gf^2]\rightarrow R\] is surjective because $X$ is defined as a closed subscheme
of $\gf^2$ (and recall that we are identifying $\gf \cong \gf^*$ via the trace pairing). Let $I$ be the kernel of this
surjection. The composite 
\[H^0(F,\Sym\left[\gf^2\right])\rightarrow R\xrightarrow{f^*}\tilde{R}=H^0(F,\Sym\left[(\gf/\bfrak)^2\right])\]
is the morphism induced by the natural surjective map of coherent sheaves $\gf^2\rightarrow (\gf/\bfrak)^2$.

Letting $S=\F[\gf^2]$ and $S_+$ be the irrelevant ideal of $S$, Proposition 2.1.5 of
\cite{snowdenSingularitiesOrdinaryDeformation2018} gives us that 
\[\Tor^S_0(R,\F)=\tilde{R}/S_+\tilde{R}\cong H^0(F, \Oc_F)\oplus H^1(F, \bfrak^2).\] Because of the (non-equivariant) isomorphism of vector bundles $\bfrak\cong \Oc(-1)^2$, this shows
that $\tilde{R}/S_+\tilde{R}$ is $1$-dimensional. Hence the composite $\F[\gf^2] \rightarrow R \rightarrow \tilde{R}$ is
surjective. It follows that the map $H^0(X,\Oc_X)\rightarrow H^0(Y,\Oc_Y)$ is an isomorphism as required.
\end{proof}

\begin{remark}
    The above proof works equally well for the Steinberg component of the fixed-determinant moduli space in the case $l = 2$.
\end{remark}

Next we turn to the main case of interest for this article, $n = 3$; we defer the actual cohomological calculations until the next section.

\begin{proof}[Proof of Theorem \ref{thm:resolution} when $n=3$]
By Proposition~\ref{propYcohom} below, \[H^i(Y,\Oc_Y)=H^i(Y,\omega_Y)=0\] for all $i>0$. To prove the theorem, the only
thing that remains to check is that the natural morphism $f^*H^0(X,\Oc_X)\hookrightarrow H^0(Y,\Oc_Y)$, injective since
$Y \to X$ is birational, is an isomorphism.    

As in the case $n = 2$, define $S=\F[\gf^2]=\Sym[\gf^2]$, with irrelevant ideal $S_+$. Let $\tilde{R}=H^0(Z,\Oc_Z)$ for
$Z$ the total space of the vector bundle $\nf^2$ on $G/B$. By Lemma~\ref{lem:Zcohom}, $H^i(Z, \Oc_Z) = 0$ for $i > 0$.
We may therefore apply Proposition 2.1.5 of \cite{snowdenSingularitiesOrdinaryDeformation2018} to deduce that 
\[\tilde{R}/S_+\tilde{R}=\Tor^S_0(\tilde{R},\F)=\bigoplus  H^i(F,\Lambda^i[\bfrak^2])[i].\] By Calculation~\ref{calc2}
below, we know that $H^i(F,\Lambda^i[\bfrak^2])=0$ unless $i=0$, when \[H^0(F,\Lambda^0[\bfrak^2])=H^0(F,\F)=\F.\] Thus
the map $S\rightarrow \tilde{R}$ is surjective and, as $H^1(Z, \Ic_Y) = 0$ by Proposition~\ref{prop:sheafcalcs} and
Lemma~\ref{lem:Zcohom}, the map $\tilde{R} = H^0(Z, \Oc_Z)\rightarrow H^0(Y,\Oc_Y)$ is surjective. The composite map 
\[S \to H^0(Z, \Oc_Z) \to H^0(Y, \Oc_Y)\]
is equal to the composite 
\[S \to H^0(X, \Oc_X) \xrightarrow{f^*}H^0(Y,\Oc_Y),\] and so $f^*$ is surjective as required. 
\end{proof}

\section{Cohomology of sheaves on the flag variety when $n=3$}
\label{sec:cohom-calc}

We let $n = 3$ and $G = \SL_3$, and continue with the notation of section~\ref{sec:buns}.
Let $X(T)$ be the character lattice, $X^{\vee}(T)$ be the cocharacter lattice, and $\langle\,,\rangle:X(T)\times X^{\vee}(T)\rightarrow \Z$ be the natural pairing.
For a character $\lambda\in X(T)$, let $\FF(\lambda)$ be the corresponding representation of $B$ and recall that in section \ref{sec:buns} we have defined a line bundle $\Oc(\lambda)$ on $G/B$. We have chosen the system $\Phi^+$ positive roots such that the weights of $\nf$ are \emph{negative}. Our basic tool will be the Borel--Weil--Bott theorem, which holds in positive characteristic for `sufficiently small' weights.

Let 
    \[\overline{C}_{\Z}=\{\lambda\in X(T):0\leq \langle\lambda+\rho,\beta^{\vee}\rangle\leq l \textrm{ for all } \beta\in\Phi^+\},\]
    \[X(T)_+ = \{\lambda \in X(T) : \langle \lambda, \beta^{\vee}\rangle \ge 0 \text{ for all } \beta \in \Phi^+\}\},\]
    and $C^0_\ZZ = \overline{C}_{\Z} \cap X(T)_+$.
Recall the `dot action' of the Weyl group $W \cong S_3$ of $G$ on $X(T)$:
\[w\cdotp \lambda = w(\lambda + \rho) - \rho.\]

\begin{theorem}\label{thm:BWB}
    Let $\lambda\in \bar{C}_{\Z}$. 
   \begin{enumerate}
     \item If $\lambda \in C^0_\ZZ$ then $H^0(G/B, \Oc(\lambda))$ is an irreducible representation of $G$ that we denote by $V(\lambda)$.  
   \item The representations $V(\lambda)$ for $\lambda \in C^0_\Z$ are pairwise non-isomorphic.
       \item If $\lambda \not \in C^0_\ZZ$, then $H^i(G/B, \Oc(\lambda)) = 0$ for all $i$.
     \item If $w\in W$, then
    $H^i(G/B,\Oc(w\cdotp\lambda))=0$ unless $i=l(w)$, in which case
    \[H^{l(w)}(G/B,\Oc(w\cdotp\lambda))\cong H^0(G/B,\Oc(\lambda)).\]
   \end{enumerate} 
\end{theorem}
\begin{proof}
   This is Corollary II.5.5 and Corollary II.5.6 of \cite{Jantzen03}.
\end{proof}

In what follows, for $V$ a representation of $B$, we will abbreviate $H^i(G/B, V)$ to $H^i(V)$. For $V$ a representation of $B$ and $\lambda$ a character of $T$ we write $V_{\lambda} = \{v \in V : tv =\lambda(v)t \text{ for all $t \in T$}\}$ for the $\lambda$-weight space of $V$. The multiset in which each character $\lambda$ of $T$ occurs $\dim(V_{\lambda})$ times is the multiset of weights of $V$; its elements are the weights of $V$.

We will call $\bigcup_{w \in W} w \cdot \bar{C}_{\Z}$ the \emph{BWB locus} and say that a representation $V$ of $B$ is \emph{BWB-good} if all of its weights are in the BWB locus. For $i\ge 0$ let $C^i_{\Z}=\coprod_{l(w)=i}w\cdotp C^0_{\Z}$. For $i \ge 0$ and $V$ a BWB-good representation of $B$, we let 
\[\psupp^i(V) = \bigcup_{l(w) = i} \{\lambda \in C^0_{\Z} : w\cdot\lambda \text{ is a weight of $V$}\},\]
viewed as a multiset where the multiplicity of $\lambda$ is the sum of the multiplicities of $w\cdot \lambda$ as weights of $V$. 

\begin{lemma}\label{lem:sup-lemma} Suppose that $V$ is a BWB-good representation of $B$ and $i \ge 0$.
\begin{enumerate}
\item  The $G$-representation $H^i(V)$ has a composition series in which all irreducible subquotients have the form $V(\lambda)$ for $\lambda \in \psupp^i(V)$, and each $V(\lambda)$ occurs with multiplicity at most the multiplicity of $\lambda$ in $\psupp^i(V)$.
\item  The representation $H^i(V)$ is completely reducible.
\end{enumerate}
\end{lemma}
\begin{proof}
Part~(1) results from repeatedly applying the long exact sequence in cohomology to a composition series for $V$. By \cite[Proposition II.4.13]{Jantzen03}, \[\Ext^i_G(V(\lambda),V(\mu)) = 0\] for all $\lambda, \mu \in C^0_{\ZZ}$ and $i > 0$. Part~(2)  follows from this and part~(1).
\end{proof}

If $V$ is a representation of $B$ then we let 
\[\chi(V) = \sum_{i \ge 0} (-1)^i[H^i(V)]\]
where the sum is taken in the Grothendieck group of the category of finite-dimensional representations of $G$. Formation of $\chi$ is additive in short exact sequences, and we therefore have 
\[\chi(V) = \sum_{i \ge 0} (-1)^i\sum_{\lambda \in \psupp^i(V)}[V(\lambda)]\]
for a BWB-good representation $V$ of $B$.

Finally, we will also need the fact (see \cite[Proposition I.4.8]{Jantzen03}) that, if $V$ is a $B$-representation and $W$ a $G$-representation, then $H^i(V\otimes W)\cong H^i(V)\otimes W$ as $G$-representations for all $i\ge 0$.

We let $L_1, L_2, L_3 \in X(T)$ be the characters taking $t \in T$ to its respective diagonal entries, labelled so that $L_1$ corresponds to the bottom right entry (!) and $L_3$ to the top left. Then $L_1$ and $-L_3$ generate the monoid $X(T)_+$ of dominant weights and we have $\rho = L_1 - L_3$. The positive simple roots are $\alpha = L_1 - L_2$ and $\beta = L_2 - L_3$. The lines $\{\mu : \pres{\mu + \rho, \kappa^\vee} = 0\}$ for $\kappa \in \{\alpha, \rho, \beta\}$ divide $X(T)\otimes \R$ into six regions. The weights strictly in the interior of each region are $w \cdot X(T)_+$ for some $w \in W$. Figure~\ref{fig:sl3} shows the BWB-locus when $p=5$ --- the interior and boundary of blue dashed region --- and shows the different $C^i_{\Z}$. Also shown are the weights $L_1,L_2,L_3$, $\alpha,\beta,\rho,-\rho$ and $-2\rho$; the weights of $\bfrak$ are coloured red and those of $\gf/\bfrak$ coloured green.

\begin{figure}[h!]
    \caption{The weight lattice of $SL_3$.}
    \label{fig:sl3}
    \includegraphics[width=0.85\textwidth]{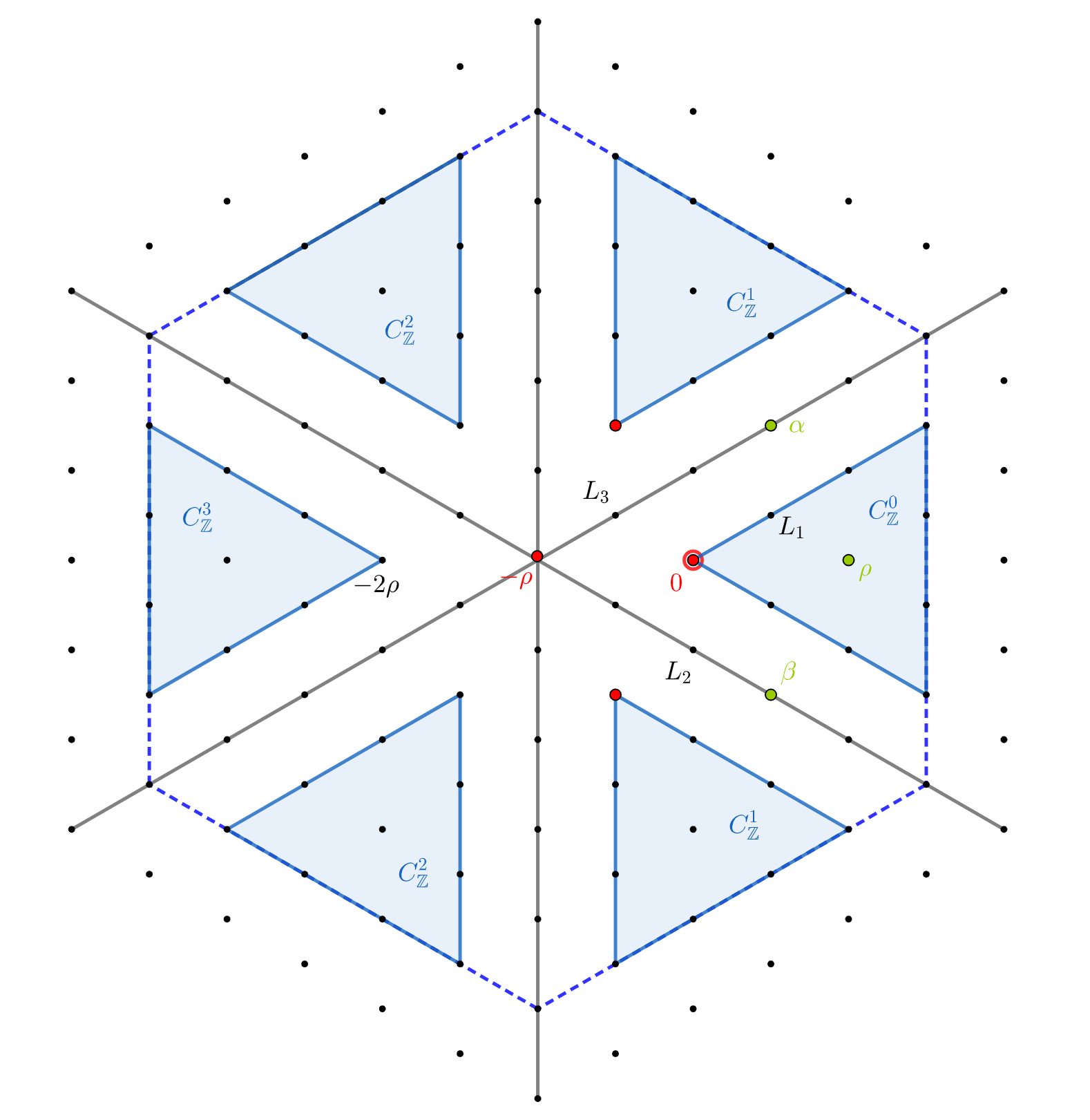}
\end{figure}

\newpage

\subsubsection{Calculating $H^i(\Lambda^j[\bfrak\oplus\bfrak])$}

\begin{calculation}{\label{calc1}}
Suppose that $l \ge 5$. Then the $G$-representations $H^i(\Lambda^j\bfrak)$ are as shown in Table~\ref{tab:calc1}.
Moreover, $H^i(\bfrak \otimes \bfrak) = H^i(\Lambda^2\bfrak)$ for all $i$.
\begin{table}[h!]
    \centering
    \begin{tabular}{c|cccc}
\diagbox{j}{i}   & $0$                    & $1$                    & $2$                    & $3$                    \\\hline
$0$   & $\F$                    & $\cdot$ & $\cdot$ & $\cdot$ \\
$1$   & $\cdot$ & $\cdot$ & $\cdot$ & $\cdot$ \\
$2$   & $\cdot$ & $\F$                     & $\cdot$ & $\cdot$ \\
$3$   & $\cdot$ & $\cdot$ & $\F$                     & $\cdot$ \\
$4$   & $\cdot$ & $\cdot$ & $\cdot$ & $\cdot$ \\
$5$   & $\cdot$ & $\cdot$ & $\cdot$ & $\F$                 
\end{tabular}
    \caption{The cohomology groups $H^i(\Lambda^j\bfrak)$.}
    \label{tab:calc1}
\end{table}
\end{calculation}
\begin{proof}
    If $l \ge 5$ then each of the representations $\bfrak$, $\gf/\bfrak$, $\bfrak \otimes \bfrak$ and $\bfrak \otimes \gf/\bfrak$ is BWB-good. We may therefore discuss the potential support of these representations and apply Lemma~4.2 to them.

    \begin{enumerate}[wide]
    \setcounter{enumi}{-1}
        
    \item When $j=0$, $\Lambda^j\bfrak=\F$. By Theorem~\ref{thm:BWB}, $H^i(\F)=0$ unless $i=0$, in which case $H^0(\F)=\F$.
    
    \item When $j=1$, we note that  
    $\psupp^2(\bfrak)=\psupp^3(\bfrak)=\emptyset$, so 
    \[H^2(\bfrak)=H^3(\bfrak)=0.\]

    We have \[\psupp^0(\bfrak) = \psupp^1(\bfrak) = \{0, 0\}\]
    and so all composition factors of $H^0(\bfrak)$ and $H^1(\bfrak)$ are trivial as $G$-representations.

    Further, there is a short exact sequence 
    \[0\rightarrow\bfrak\to \gf\to \gf/\bfrak\to 0\]
    which gives a long exact sequence in cohomology
    
    \[\begin{tikzcd}
	0 && {H^0(\bfrak)} & {} & {H^0(\gf)} && {H^0(\gf/\bfrak)} \\
	&& {H^1(\bfrak)} && {H^1(\gf)} && \hdots
	\arrow[from=1-3, to=1-5]
	\arrow[from=1-5, to=1-7]
	\arrow[from=1-1, to=1-3]
	\arrow[from=1-7, to=2-3,out=0, in=180,looseness=1]
	\arrow[from=2-3, to=2-5]
	\arrow[from=2-5, to=2-7]
\end{tikzcd}\]
    As $\gf$ is a $G$-representation, we see from part~(0) that $H^0(\gf)=\gf$ and $H^1(\gf)=0$. For $\gf/\bfrak$ we have $\psupp^0(\gf/\bfrak)=\{\rho\}$ and $\psupp^i(\gf/\bfrak) = 0$ for $i > 0$. Thus $H^0(\gf/\bfrak)\subset V(\rho) = \gf$. As all subquotients of $H^0(\bfrak)$ and $H^1(\bfrak)$ are the trivial representation, we must have $H^0(\bfrak) = H^1(\bfrak) = 0$ and $H^0(\gf/\bfrak) = \gf$.
    
    \item When $j=2$, consider the exact sequence
    \[0\rightarrow\bfrak\otimes\bfrak\to \gf\otimes\bfrak\to \gf/\bfrak\otimes\bfrak\to 0,\]
    from which we get a long exact sequence with parts 
    \[H^i(\gf\otimes\bfrak)\to H^i(\gf/\bfrak\otimes\bfrak)\to H^{i+1}(\bfrak\otimes\bfrak)\to H^{i+1}(\gf\otimes\bfrak).\]
    As $H^i(\bfrak)=0$ for all $i$, we obtain isomorphisms $H^i(\bfrak\otimes \gf/\bfrak)\cong H^{i+1}(\bfrak\otimes \bfrak)$ for all $i$. Thus, $H^0(\bfrak\otimes\bfrak)=0$ and $H^3(\bfrak\otimes \bfrak)=H^2(\gf/\bfrak\otimes\bfrak)=0$, since $\psupp^2(\gf/\bfrak\otimes\bfrak)=\emptyset$.

    To show that $H^2(\bfrak\otimes \bfrak) = 0$, consider first the exact sequence
    \[0\rightarrow\bfrak\otimes\gf/\bfrak\to \gf\otimes\gf/\bfrak\to \gf/\bfrak\otimes\gf/\bfrak\to 0\]
    giving rise to the long exact sequence
    \[0\to H^0(\bfrak\otimes\gf/\bfrak)\to \gf \otimes \gf\to H^{0}(\gf/\bfrak\otimes\gf/\bfrak)\to H^{1}(\bfrak\otimes\gf/\bfrak)\to 0.\]
    Here we have used that $H^1(\gf\otimes \gf/\bfrak)\cong \gf \otimes H^1(\gf/\bfrak)=0$, as $\psupp^1(\gf/\bfrak) = \emptyset$, and $H^0(\gf\otimes \gf/\bfrak)\cong \gf \otimes H^0(\gf/\bfrak)\cong\gf\otimes\gf$.
    Thus \[H^2(\bfrak \otimes \bfrak) \cong H^1(\bfrak\otimes \gf/\bfrak)\cong\coker(\gf\otimes\gf\to H^0(\gf/\bfrak\otimes \gf/\bfrak)).\]

    Since $\psupp^2(\bfrak \otimes \bfrak) = \{0,0\}$, every irreducible constituent of $H^2(\bfrak\otimes \bfrak)$ is trivial. So it is enough to show that $H^0(\gf/\bfrak\otimes \gf/\bfrak)$ does not have the trivial representation as a subquotient. If $l \ge 7$, this follows as $\gf/\bfrak \otimes \gf/\bfrak$ is BWB-good and $0 \not \in \psupp^0(\bfrak \otimes \gf/\bfrak)$. If $l = 5$ then every weight of $\gf/\bfrak \otimes \gf/\bfrak$ \emph{except} for $2\rho$ is in the BWB locus and nonzero, so cannot give rise to a trivial subquotient of $H^0(\gf/\bfrak \otimes \gf/\bfrak)$. Every other composition factor of $H^0(\gf/\bfrak\otimes \gf/\bfrak)$ is isomorphic to a composition factor of $H^0(\F(2\rho))$. By the strong linkage principle of \cite[Proposition II.6.13]{Jantzen03}, these are of the form $V(\rho)$ or $L(2\rho)$ (the irreducible representation with highest weight $2\rho$), and are therefore nontrivial as required.

     Finally, we compute 
     \[\chi(\bfrak\otimes\bfrak)=4[\FF]-4[\FF]-4[\FF]+2[\F]+2[\F]-[\F]=-[\F].\] Since $H^i(\bfrak\otimes \bfrak) = 0$ for $i \ne 1$, we deduce that $H^1(\bfrak\otimes \bfrak)=\F$. As $\Lambda^2\bfrak$ is a direct summand of $\bfrak\otimes \bfrak$ (as $l > 2$), it follows that $H^i(\Lambda^2\bfrak)=0$ when $i\neq 1$ and, as 
    \[\chi(\Lambda^2\bfrak)=[\F]-4[\F]+2[\F]=-[\F],\]
    $H^1(\Lambda^2\bfrak)= \F = H^1(\bfrak \otimes \bfrak)$.

    For later use, also note that we have shown $H^i(\bfrak \otimes \gf/\bfrak) = 0$ for $i > 0$; an Euler characteristic calculation then shows that $H^0(\bfrak \otimes \gf/\bfrak) = \F$.

    \item For $j = 3, 4, 5$ we have a $B$-equivariant pairing 
    \[\Lambda^j\bfrak \times \Lambda^{5-j}\bfrak \to \Lambda^5 \bfrak \cong \F(-2\rho)\]
    and so \[(\Lambda^j\bfrak)^* \otimes \F(-2\rho) \cong \Lambda^{5-j}\bfrak.\]
    Since $\omega_{G/B} \cong \F(-2\rho)$, Serre duality gives
    \[H^i(\Lambda^j\bfrak) \cong H^{3-i}(\Lambda^{5-j}\bfrak)^*\]
    and the result follows from parts 0--2.\qedhere
\end{enumerate}
\end{proof}

\begin{calculation}{\label{calc2}}
Suppose that $l \geq 5$. Then
\[H^i(\bfrak \otimes \Lambda^2\bfrak) = \begin{cases}
    \gf^2 \oplus \F & \text{for $i=2$} \\
    0 & \text{otherwise.}
\end{cases}\]
\end{calculation}

\begin{proof}
Firstly, we observe that $\bfrak\otimes \bfrak \otimes \gf$, and therefore any subquotient of it, is BWB-good for $l \ge 5$. This justifies the use of potential supports in the following calculation.

Considering the long exact sequence associated to
    \[0\to\Lambda^2\bfrak\otimes \bfrak\to \Lambda^2\bfrak\otimes\gf\to\Lambda^2\bfrak\otimes \gf/\bfrak\to 0\]
    along with the isomorphism $H^i(\gf\otimes \Lambda^2\bfrak)=\gf\otimes H^i(\Lambda^2\bfrak)$ and the results of Calculation~\ref{calc1}, we obtain an exact sequence 
    \[0\to H^0(\Lambda^2\bfrak\otimes\gf/\bfrak)\to H^1(\Lambda^2\bfrak\otimes\bfrak)\to \gf \to H^1(\Lambda^2\bfrak\otimes\gf/\bfrak)\to H^2(\Lambda^2\bfrak\otimes\bfrak)\to 0,\]
    and isomorphisms
    \[H^2(\Lambda^2\bfrak\otimes\gf/\bfrak)\cong H^3(\Lambda^2\bfrak\otimes\bfrak)\] and $H^0(\Lambda^2\bfrak\otimes\bfrak)=0$. 
    As $\psupp^2(\Lambda^2\bfrak\otimes\gf/\bfrak)=\emptyset$, we see that $H^3(\Lambda^2\bfrak\otimes\bfrak)=0$. 

    We calculate (using superscripts to denote multiplicity)
    \[\psupp^i(\Lambda^2\bfrak\otimes\bfrak) = \begin{cases} \{0^2\} & $i = 0$ \\
     \{0^{10}\} & i=1 \\
     \{0^{14}, \rho^2\} & i=2 \\
     \{0^7\} & i=3
    \end{cases}\] and so 
    \[\chi(\Lambda^2\bfrak \otimes \bfrak) = 2[\gf] - [\F].\]
    Since $\rho$ appears only in $\psupp^2$, we obtain that
    \[H^1(\Lambda^2\bfrak\otimes\bfrak) \cong \F^n\]
    and
    \[H^2(\Lambda^2\bfrak\otimes\bfrak)\cong\gf^2\oplus \F^{n+1}\] for some integer $n \ge 0$. It follows that 
    \[H^0(\Lambda^2\bfrak \otimes \gf/\bfrak)\cong  \F^n\]
    since the map $H^1(\Lambda^2\bfrak\otimes\bfrak)\to \gf$ must be zero.

    The $B$-representation  $\Lambda^2\bfrak\otimes\gf/\bfrak$ is a direct summand of  $\bfrak\otimes \bfrak\otimes\gf/\bfrak$, which fits into the short exact sequence 
    \[0\to\bfrak\otimes \bfrak\otimes \gf/\bfrak\to \gf\otimes \bfrak\otimes\gf/\bfrak\to\gf/\bfrak\otimes \bfrak\otimes \gf/\bfrak\to 0\]
    and so we have an exact sequence
    \[0 \to H^0(\bfrak\otimes \bfrak\otimes \gf/\bfrak)\into \gf\otimes H^0(\bfrak\otimes\gf/\bfrak)\cong \gf\]
    (the last isomorphism appearing in part~2 of the proof of Calculation~\ref{calc1}).
    As $H^0(\Lambda^2\bfrak\otimes \gf/\bfrak)\cong \F^n$ and $\F$ is not a subrepresentation of $\gf$, we deduce that $H^0(\Lambda^2\bfrak\otimes \gf/\bfrak)=H^1(\Lambda^2\bfrak\otimes \bfrak)=0$. \qedhere
\end{proof}
\begin{calculation}\label{calc:wedge4} \begin{enumerate}\item Suppose that $l \ge 5$. Then 
\[H^3(\Lambda^3\bfrak \otimes \bfrak) = 0.\] 
\item Suppose that $l \ge 5$. Then 
\[H^3(\Lambda^2 \bfrak \otimes \Lambda^2 \bfrak) = 0.\]
\end{enumerate}
\end{calculation}
\begin{proof}\begin{enumerate}[wide] \item We first show that $H^3(\Lambda^3\bfrak \otimes \bfrak) = 0$. By Serre duality, the pairing $\Lambda^3 \bfrak \times \Lambda^2\bfrak \to \Lambda^5\bfrak \cong \F(-2\rho)$, and the isomorphism $\omega_F \cong \Oc_F(-2\rho)$, we have 
\[H^3(\Lambda^3\bfrak \otimes \bfrak) \cong H^0(\Lambda^2\bfrak \otimes \bfrak^*)^*.\]
Since $\bfrak^* \cong \gf/\nf$, we wish to show that $H^0(\Lambda^2\bfrak \otimes \gf/\nf) = 0$. 
The last line of the proof of Calculation~\ref{calc2} shows that $H^0(\Lambda^2\bfrak \otimes \gf/\bfrak) = 0$, and by Calculation~\ref{calc1} $H^0(\Lambda^2\bfrak) = 0$. 
The result follows by tensoring the short exact sequence of $B$-representations
\[0 \to \bfrak/\nf \cong \F^2 \to \gf/\nf \to \gf/\bfrak \to 0\]
with $\Lambda^2\bfrak$ and applying $H^0(\cdot)$.
    
\item 
If $l\geq 5$ then $\Lambda^2\bfrak \otimes \Lambda^2\bfrak$ is BWB-good, which justifies the following argument in characteristic $l$. Note that the only weights in $\psupp^3(\Lambda^2\bfrak \otimes \Lambda^2\bfrak)$ are $0$ and $\rho$, so that $H^3(\Lambda^2\bfrak \otimes \Lambda^2\bfrak)$ is a direct sum of copies of $\F$ and $\gf$.

Let $e_{\pm\alpha}$, $e_{\pm\rho}$ and $e_{\pm\beta}$ be root vectors in $\gf$ (chosen to be elementary matrices) and let $P$ be the parabolic subgroup of $G$ whose Lie algebra is spanned by $\bfrak$ and $e_{\alpha}$, with unipotent radical $U_P$ and Levi quotient $M = P/U_P \cong GL_2$. We will use the following lemma of Demazure.
\begin{lemma}[``Easy Lemma'' of \cite{demazureVerySimpleProof1976}]\label{lem:dem}If $V$ is a representation of $B$ and $\lambda \in X(T)$ is such that $\pres{\lambda, \alpha^\vee} = 1$, and if $V(\lambda)$ extends to a representation of $P$,  then $H^i(V) = 0$ for all $i \ge 0$.
\end{lemma}

The next lemma allows us to recognise certain representations that extend to $P$. 
\begin{lemma}\label{lem:extend} Suppose that $V$ is a representation of $B$ such that:
\begin{itemize}
    \item $\dim(V) \le l$
    \item $e_{-\beta}$ and $e_{-\rho}$ act as zero on $V$;
    \item There is a vector $v \in V$ of weight $\mu$ such that $\pres{\mu, \alpha} = \dim(V) - 1$ and such that 
    \[v, e_{-\alpha}v, \ldots, e_{-\alpha}^{\dim(V) -1}v\]
    is a basis of $V$.
\end{itemize}
Then $V$ extends to a representation of $P$.
\end{lemma}
\begin{proof}Since $e_{-\beta}$ and $e_{-\rho}$ act as zero on $V$ and $\dim(V) \le l$, the root subgroups $U_{-\beta}$ and $U_{-\rho}$ act trivially on $V$, and therefore so does $U_P$. Thus the action of $B$ factors through the quotient $B \to B/U_P$. The third condition, together with the restriction $\dim(V) \le l$, implies that $V$ is the restriction to $B/U_P$ of an irreducible representation of $M \cong GL_2$ whose restriction to $M^{\mathrm{der}} \cong SL_2$ is the irreducible representation of highest weight $\dim(V)-1$. In particular, $V$ extends to $P$.
\end{proof}

We first show that $H^3(\Lambda^2\bfrak \otimes \Lambda^2\bfrak)$ does not contain the trivial representation.
Let $W_1 \subset \Lambda^2\bfrak \otimes \Lambda^2\bfrak$ be the $\F$-span of all weight vectors of weights in
\[\{-\rho - 2\alpha, -2\rho-2\alpha,-2\rho-\alpha,-3\rho,-2\rho-\beta,-2\rho-2\beta, -\rho-2\beta\}\]
and let $W_2$ be the $\F$-span of $W_1$ and all weight vectors of weights in $\{-2\rho, -\rho - \alpha, -\rho-\beta\}$. Then $W_1$ and $W_2$ are $B$-subrepresentations of $W_3 = \Lambda^2\bfrak \otimes \Lambda^2\bfrak$. Let $V = W_2/W_1$. Because $0$ is not contained in $\psupp^3(W_1)$ or $\psupp^3(W_3/W_2)$, to show that $H^3(W_3)$ contains no copy of the trivial representation it is enough to show that $H^3(V) = 0$.

Consider the weight space $V_{-\rho-\beta}\subseteq V$. The $B$ representation generated by $V_{\rho-\beta}$ is $V^{\beta}:=V_{-\rho-\beta}\oplus e_{-\alpha}(V_{-\rho-\beta})$, and this representation splits as a direct sum into copies of $\F(-\rho-\beta)$ and $2$-dimensional representations satisfying the hypotheses of Lemma~\ref{lem:extend}. Thus, $H^3(V^{\beta})=0$ by Lemma~\ref{lem:dem}. Analogously, consider the $B$-subrepresentation $V^{\alpha}\subseteq V/V^{\beta}$ generated by $V_{-\rho-\alpha}$. In the exact same way  (switching the roles of $\alpha$ and $\beta$ in Lemmas~\ref{lem:dem} and~\ref{lem:extend}), we see that $H^3(V^{\alpha})=0$. To show that $H^3(V)=0$, it remains simply to show that $V^{\alpha}=V/V^{\beta}$, or equivalently, that $V_{-2\rho}=e_{-\alpha}(V_{-\rho-\beta})+e_{-\beta}(V_{-\rho-\alpha})$.

This can be calculated via a direct calculation by showing that each of the $17$ simple tensors which span $V_{-2\rho}$ are in this joint image. We outline this calculation. Set $f_{\gamma}=e_{-\gamma}\in \bfrak$ (reserving $e_{-\gamma}=[f_{\gamma},\textunderscore]$ for the linear map) and set $t_{\alpha},t_{\beta}\in\tf$ a basis of weight $0$ vectors such that $e_{-\nu}(t_{\mu})=\delta_{\nu,\mu}f_{\nu}$ in $\bfrak$ for $\nu,\mu\in\{\alpha,\beta\}$. (Such a basis exists when the characteristic $l\neq 3$). Then
\begin{align*}
    (f_{\alpha}\wedge f_{\rho})\otimes(t_{\alpha}\wedge f_{\beta})&=e_{-\beta}\Big[(f_{\alpha}\wedge f_{\rho})\otimes(t_{\alpha}\wedge t_{\beta})\Big] \\
    (f_{\alpha}\wedge f_{\rho})\otimes(t_{\beta}\wedge f_{\beta})&=\frac{1}{2}e_{-\alpha}\Big[(t_{\alpha}\wedge f_{\rho})\otimes(t_{\beta}\wedge f_{\beta}) \\ &
    \textcolor{red}{+(f_{\beta}\wedge f_{\alpha})\otimes(t_{\beta}\wedge f_{\beta})}
    +(f_{\beta}\wedge t_{\alpha})\otimes(t_{\beta}\wedge f_{\rho})\Big] \\
    (t_{\alpha}\wedge f_{\rho})\otimes(f_{\alpha}\wedge f_{\beta})&=\frac{1}{2}e_{-\beta}\Big[\textcolor{red}{(t_{\alpha}\wedge f_{\alpha})\otimes(t_{\beta}\wedge f_{\rho})} \\ &
    +(t_{\alpha}\wedge f_{\alpha})\otimes(f_{\alpha}\wedge f_{\beta})
    \textcolor{red}{+(t_{\alpha}\wedge f_{\rho})\otimes(f_{\alpha}\wedge t_{\beta})}\Big].
\end{align*}
These three equations give us $12$ linearly independent basis elements in the image when we consider the two independent symmetries of swapping the left and right sides of the tensor, and the symmetry interchanging $e_{-\alpha}$ and $-e_{-\beta}$, the effect of which is to flip the signs of the terms in red. We have further:
\begin{align*}
    (t_{\alpha}\wedge f_{\rho})\otimes(t_{\alpha}\wedge f_{\rho})&=e_{-\beta} \Big[\textcolor{red}{(t_{\alpha}\wedge f_{\alpha})\otimes(t_{\alpha}\wedge f_{\rho})}\Big] \\
    (t_{\alpha}\wedge f_{\rho})\otimes(t_{\beta}\wedge f_{\rho})&=\frac{1}{2}e_{-\beta}\Big[\textcolor{red}{(t_{\alpha}\wedge f_{\alpha})\otimes(t_{\beta}\wedge f_{\rho})} \\ &
    +(t_{\alpha}\wedge f_{\alpha})\otimes(f_{\alpha}\wedge f_{\beta})
    \textcolor{red}{-(t_{\alpha}\wedge f_{\rho})\otimes(f_{\alpha}\wedge t_{\beta})}\Big]
\end{align*}
again with the two conjugates (swapping $\alpha\leftrightarrow \beta$ and flipping the signs of the terms in red) giving us in total $16$ of the simple tensor spanning $V_{-2\rho}$.

The last basis element can be expressed as 
\begin{align*}
(f_{\alpha}\wedge f_{\beta})\otimes (f_{\alpha}\wedge f_{\beta})&=e_{-\beta}\Big[(f_{\alpha}\wedge f_{\beta})\otimes(f_{\alpha}\wedge t_{\beta})\Big] \\ &-(f_{\alpha}\wedge f_{\beta})\otimes (f_{\rho}\wedge t_{\beta})-(f_{\rho}\wedge f_{\beta})\otimes(f_{\alpha}\wedge t_{\beta}) 
\end{align*}
with the last two terms inside $e_{-\alpha}V_{-\rho-\beta}+e_{-\beta}(V_{\rho-\alpha})$ as previously considered. Since $V_{-2\rho}$ is a $17$ dimensional vector space, this shows that $V_{-2\rho}=e_{-\alpha}(V_{-\rho-\beta})+e_{-\beta}(V_{-\rho-\alpha})$ and, consequently, $H^3(V)=0$.

We have now shown that $H^3(\Lambda^2\bfrak \otimes \Lambda^2\bfrak)\cong \gf^s$ for some $s \ge 0$. The only weight of $\Lambda^2\bfrak \otimes \Lambda^2\bfrak$ remaining that can contribute to $H^3$ is $-3\rho$, which occurs with multiplicity two; thus $s \le 2$ and we wish to show that $s = 0$.

Let $\mu = -2\beta - \rho$. In $\Lambda^2 \bfrak \otimes \Lambda^2\bfrak$ let 
\[v = (e_{-\rho}\wedge e_{-\beta}) \tensor (e_{-\beta}\wedge t_{\alpha})\]
and 
\[v' = (e_{-\beta}\wedge t_{\alpha})\tensor (e_{-\rho}\wedge e_{-\beta}).\]
Let $\tilde{V}$ be the $B$-representation generated by $v$, and let $V = \tilde{V}_{U_P}$ be its $U_P$-coinvariants. The weights of $V$ are a subset of 
\[\{\mu = -2\beta - \rho, \mu - \alpha, \mu - 2\alpha = -3\rho\}. \] Since 
\[e_{-\alpha}^2(v) = 2(e_{-\rho} \wedge e_{-\beta}) \otimes (e_{-\rho} \wedge e_{-\alpha})\ne 0\]
and $\pres{\mu + L_1, \alpha} = 2$, Lemma~\ref{lem:extend} applies to show that $V(L_1)$ extends to a representation of $P$ and hence that $H^i(V) = 0$ for all $i$. 

Now, $\psupp^3(\ker(\tilde{V}\rightarrow V)) = \emptyset$ and so, as $H^i(V) = 0$ for all $i$,
\[H^3(\tilde{V}) = 0.\]
Constructing $\tilde{V}'$ similarly from $v'$ we see that $H^3(\tilde{V}') = 0$. Let 
\[Q = (\Lambda^2\bfrak \otimes \Lambda^2\bfrak)/(\tilde{V} + \tilde{V}').\]
Then \[H^3(Q) \cong H^3(\Lambda^2\bfrak \otimes \Lambda^2\bfrak) \cong \gf^s.\]
However, $\tilde{V} + \tilde{V}'$ contains the entire two-dimensional weight space of $\Lambda^2\bfrak \otimes \Lambda^2\bfrak$ of weight $-3\rho$, and so $\rho \not \in \psupp^3(Q)$. Thus $s = 0$, as required. \qedhere
  
\end{enumerate}
\end{proof}
\begin{calculation}{\label{calc3}}
    Suppose that $l \ge 5$, $0 \le i \le 3$, and $0 \le j \le 4$. 
    Then the cohomology group $H^i(\Lambda^j(\bfrak \oplus \bfrak))$ is as given in Table~\ref{tab:calc3}. 
    \begin{table}[h!]
        \centering
\begin{tabular}{l|cccc}
 \diagbox{j}{i} & 0       & 1       & 2               & 3       \\ \hline
 0 & $\F$     & $\cdot$ & $\cdot$         & $\cdot$ \\
 1 & $\cdot$ & $\cdot$ & $\cdot$         & $\cdot$ \\
 2 & $\cdot$ & $\F^3$   & $\cdot$         & $\cdot$ \\
 3 & $\cdot$ & $\cdot$ & $\gf^4\oplus \F^4$ & $\cdot$ \\
 4 & ? & ? & ? & $\cdot$ \\
\end{tabular}
\caption{The group $H^i(\Lambda^j[\bfrak \oplus \bfrak])$. Cohomology groups which are zero are denoted by a dot, while those which are unknown are denoted by a question mark.}
              \label{tab:calc3}
    \end{table}
    
\end{calculation}

\begin{proof}
This is immediate from Calculations~\ref{calc1}, \ref{calc2} and~\ref{calc:wedge4} and the decompositions
\begin{align*}
    \Lambda^2(\bfrak \oplus \bfrak) &= \left(\Lambda^2\bfrak\right)^2 \oplus \bfrak \otimes \bfrak,\\
\Lambda^3(\bfrak \oplus \bfrak) &= \left(\Lambda^3 \bfrak\right)^2 \oplus \left(\Lambda^2\bfrak \otimes \bfrak\right)^2\\ \intertext{and}
\Lambda^4(\bfrak \oplus \bfrak) &= \left(\Lambda^4 \bfrak\right)^2 \oplus \left(\Lambda^3\bfrak \otimes \bfrak\right)^2 \oplus \left(\Lambda^2\bfrak \otimes \Lambda^2\bfrak\right)\\
\end{align*}
which hold for $l \ge 5$.\qedhere
\end{proof}

\begin{remark}\label{rmk:hemelsoet-voorhaar} It appears that Calculation~\ref{calc:wedge4} is near the limit of what our \emph{ad hoc} methods can handle.
    Using the programs of Hemelsoet and Voorhaar \cite{hemelsoetComputerAlgorithmBGG2021}, available at \url{https://github.com/RikVoorhaar/bgg-cohomology}, it is easy to verify all of the previous calculations over a field of characteristic zero. This implies that these calculations are correct in sufficiently large positive characteristic; however, it is not clear to us how to make this effective (their algorithm relies on the BGG resolution).
\end{remark}

\subsubsection{Twists}

\begin{calculation}\label{calc:twists} Let $\lambda \in X(T)_+$ be a nonzero dominant weight. Suppose that $l \ge 5$. Then:
\begin{enumerate}\setcounter{enumi}{-1}
    \item $H^i(\F(\lambda)) = 0$ for all $i > 0$;
    \item $H^i(\bfrak(\lambda)) = 0$ if $i > 1$, and $H^1(\bfrak(\lambda)) = 0$ if $\lambda \ge \rho$ (in the usual partial order on weights) or if $l \ge 2 + \max(\pres{\lambda, \alpha^\vee}, \pres{\lambda, \beta^\vee})$;
    \item $H^i(\Lambda^2[\bfrak \oplus \bfrak](\lambda)) = 0$ if $i \ge 2$;
    \item $H^3(\Lambda^3[\bfrak \oplus \bfrak](\lambda)) = 0$.
\end{enumerate}

\end{calculation}
\begin{proof}
    Recall that $H^i(\F(\lambda)) = 0$ for all $i > 0$ by Kempf's vanishing theorem \cite[Proposition II.4.5]{Jantzen03}. This deals immediately with part~(0). Recall also from \cite[Proposition II.5.4]{Jantzen03} that, if $\mu \in X(T)$ with $\pres{\mu, \kappa^\vee} = -1$ for some simple root $\kappa$, then $H^i(\F(\mu)) = 0$ for all $i \ge 0$.

    For part~(1), if $\lambda \geq \rho$ then every weight $\mu$ of $\bfrak(\lambda)$ lies in $\bar{C}_\Z$, and we conclude that $H^i(\bfrak(\lambda)) = 0$ for all $i > 0$. Otherwise, we have (without loss of generality) that $\lambda = aL_1$; assume this now.

     For every weight $\mu$ of $\bfrak(\lambda)$, we have $\mu \in \bar{C}_\Z$ (in which case $H^i(\F(\mu)) = 0$ for all $i > 0$), or $\mu \in C^1_\Z$ and $\pres{\mu + \rho, \beta^\vee} = -1$. By \cite[Proposition II.5.4~(d)]{Jantzen03}, we have \[H^i(\F(\mu)) \cong H^{i-1}(\F(s_\alpha\cdot \mu))= 0\] for all $i > 1$. It follows that $H^i(\bfrak(\lambda)) = 0$ for all $i\ge 2$.

    Suppose now that $l \ge 2 + a$. Then $\bfrak(\lambda)$ is BWB-good and $\psupp^1(\bfrak(\lambda)) = \{\lambda\}$. On the other hand, from the short exact sequence 
\[0 \to \bfrak(\lambda) \to \gf(\lambda) \to (\gf/\bfrak)(\lambda) \to 0\]
we obtain a surjection $H^0((\gf/\bfrak)(\lambda)) \onto H^1(\bfrak(\lambda))$. But, using that $\lambda = aL_1$, $\psupp^0((\gf/\bfrak)(\lambda)) = \{\rho + \lambda, \beta + \lambda\}$ which is disjoint from $\{\lambda\}$, and so
\[H^1(\bfrak(\lambda)) = 0\]
as required.

    For part~(2), as in the paragraph before last, it is enough to consider weights $\mu$ of $\Lambda^2[\bfrak \oplus \bfrak](\lambda)$ with $\pres{\mu + \rho, \alpha^\vee} < 0$. Any such weight $\mu$ either lies in $s_\alpha \cdot \bar{C}_\Z$, or is equal to $-2\rho - L_3$ and so lies on the line $\pres{\mu + \rho, \beta^\vee} = 0$. In the latter case, $H^i(\F(\mu)) = 0$ for all $i$; in the former, as $\pres{\mu + \rho, \alpha^\vee} \ge -3$ and $l \ge 3$, \cite[Proposition II.5.4~(d)]{Jantzen03} implies that 
    \[H^i(\F(\mu)) \cong H^{i-1}(\F(s_\alpha \cdot \mu)) = 0\]
    for all $i \ge 2$. Thus $H^i(\Lambda^2[\bfrak \oplus \bfrak](\lambda)) = 0$ for all $i \ge 2$.

Finally, for part~(3), every weight $\mu$ of $\Lambda^3[\bfrak\oplus \bfrak](\lambda)$ lies in some $w\cdot\bar{C}_\Z$ for $l(w) \le 2$ and satisfies $\pres{\mu + \rho, \kappa^\vee} \le 4 < l$ for each negative root $\kappa$. We can therefore apply \cite[Proposition II.5.4~(d)]{Jantzen03} $l(w)$ times and deduce that $H^3(\F(\mu)) = 0$. Thus $H^3(\Lambda^3[\bfrak \oplus \bfrak](\lambda)) = 0$.
\end{proof}

\begin{calculation}\label{calc:alpha-twists} Let $\lambda \in X(T)_+$. If $l \ge 3$ then 
\[H^i(\Lambda^j[\bfrak\oplus \bfrak](\lambda + \alpha)) =H^i(\Lambda^j[\bfrak\oplus \bfrak](\lambda + \beta))=  0\]
whenever $i > j$.

If $l \ge 7$ and $i \ne 1$, then
\[H^i(\Lambda^i[\bfrak \oplus\bfrak](\alpha)) = H^i(\Lambda^i[\bfrak \oplus\bfrak](\beta))= 0,\]
while
\[H^1(\bfrak(\alpha)) \cong H^1(\bfrak(\beta)) \cong \gf.\]
\end{calculation}
\begin{proof}
    The statement in the case $i > j$ can be proved by applying \cite[Proposition II.5.4]{Jantzen03} exactly as in Calculation~\ref{calc:twists}, and we omit the details. We therefore focus on the case $i = j$; by symmetry, it is enough to deal with $\alpha$. The reader may check that, in the calculations below, every representation to which we apply $\psupp$ is BWB-good for $l\ge 7$.

    If $i = 0$ then $H^0(\F(\alpha)) = 0$ as $\alpha$ is on the boundary of $\bar{C}_\Z$.

    We have already shown that $H^i(\bfrak(\alpha)) = 0$ for $i > 1$.  From the inclusion $\bfrak \into \gf$ we obtain an inclusion
    \[H^0(\bfrak(\alpha)) \into \gf\otimes H^0(\F(\alpha)) = 0\]
    and so $H^0(\bfrak(\alpha)) = 0$. Computing the Euler characteristic shows that $H^1(\bfrak(\alpha)) \cong \gf$.

    We next show that $H^1((\gf/\bfrak)(\alpha)) = 0$. Indeed, this sits in a short exact sequence
    \[\gf\otimes H^1(\F(\alpha)) \to H^1((\gf/\bfrak)(\alpha)) \to H^2(\bfrak(\alpha))\]
    and the outer terms are both zero.

    Now, from the exact sequence \[0 \to \bfrak \otimes \bfrak(\alpha) \to \bfrak \otimes \gf(\alpha)\to (\bfrak \otimes \gf/\bfrak)(\alpha) \to 0\] and the fact that $H^2(\bfrak(\alpha)) = 0$, we obtain a surjection
    \[H^1((\bfrak \otimes \gf/\bfrak)(\alpha)) \onto H^2(\bfrak \otimes \bfrak(\alpha))\]
    As $\psupp^2(\bfrak\otimes \bfrak(\alpha)) = \{0\}$ we have that $H^2(\bfrak\otimes \bfrak(\alpha)) \subset \F$. 

    From the short exact sequence $0 \to (\bfrak \otimes \gf/\bfrak)(\alpha) \to (\gf \otimes \gf/\bfrak)(\alpha)\to (\gf/\bfrak \otimes \gf/\bfrak)(\alpha) \to 0$ and the vanishing of $H^1((\gf/\bfrak)(\alpha))$ we get a surjection
    \[H^0((\gf/\bfrak \otimes \gf/\bfrak)(\alpha)) \onto H^1(\gf/\bfrak \otimes \bfrak(\alpha))\]
    and hence a surjection
    \[H^0((\gf/\bfrak \otimes \gf/\bfrak)(\alpha)) \onto H^2(\bfrak \otimes \bfrak(\alpha)).\]
    But $\psupp^0(\gf/\bfrak \otimes \gf/\bfrak(\alpha))$ does not contain $\{0\}$, so we must have
    \[H^2(\bfrak \otimes \bfrak(\alpha))= 0\]
    from which we can obtain as in Calculation \ref{calc3}
    \[H^2(\Lambda^2[\bfrak\oplus \bfrak](\alpha))=0.\]

    Finally we do the case $i = j = 3$. Here $\psupp^3(\Lambda^3\bfrak(\alpha)) = \emptyset$ so it is enough to show that
    \[H^3(\Lambda^2\bfrak \otimes \bfrak(\alpha)) = 0.\]
    From the usual exact sequence and the vanishing of $H^3(\Lambda^2\bfrak(\alpha))$ we obtain a surjection
    \[H^2((\Lambda^2\bfrak \otimes \gf/\bfrak)(\alpha)) \onto H^3(\Lambda^2\bfrak \otimes \bfrak(\alpha)).\]
    But $\psupp^2(\Lambda^2\bfrak \otimes \gf/\bfrak(\alpha)) = \emptyset$ and so
    \[H^2((\Lambda^2\bfrak \otimes \gf/\bfrak)(\alpha)) = H^3(\Lambda^2\bfrak \otimes \bfrak(\alpha)) = 0.\qedhere\]
    \end{proof}

\subsection{Results for the cohomology of sheaves on $Z$ and $Y$.}

We now use the above results to calculate cohomology groups of particular coherent sheaves on $Z$ and $Y$.

\begin{lemma}\label{lem:Zcohom}
    Let $\lambda \in X(T)_+$ and suppose that $l \ge 5$. 
    Then \[H^i(Z,\Oc_Z(\lambda)) = H^i(Z, \Oc_Z(\lambda + \alpha)) = H^i(Z, \Oc_Z(\lambda + \beta))=0\] for all $i>0$.
\end{lemma}

\begin{proof}
    Because the map $\pi_Z:Z\rightarrow F$ is affine, from Proposition~\ref{prop:sheafcalcs} we have the following isomorphisms of cohomology groups:
    \[H^i(Z,\Oc_Z(\lambda))=H^i(F,\pi_{Z,*}[\Oc_Z(\lambda)])=H^i(F, \Sym([\gf/\bfrak]^2)(\lambda))\]
    So the question reduces to the calculation of the groups $H^i(F,\Sym^n[\gf/\bfrak^2])$.

    As in section 4 of \cite{VilonenXue}, one can consider the Koszul complex of $0\to \bfrak^2 \to \gf^2\to \gf/\bfrak^2\to 0$ and twist by $\lambda$, giving us a resolution
\[...\to \Lambda^i[\bfrak^2]\otimes \Sym^{n-i}[\gf^2](\lambda) \to ...\to\Sym^n[\gf^2](\lambda)\to \Sym^n[\gf/\bfrak^2](\lambda)\to 0\]

 We therefore have a spectral sequence
 \[E_1^{-i,i+k} = H^{i+k}(\Lambda^i[\bfrak^2](\lambda))\otimes \Sym^{n-i}[\gf^2]\Longrightarrow H^k(\Sym^n[\gf/\bfrak^2](\lambda)).\]
 
By Calculation~\ref{calc3} in the case of $\lambda=0$, and Calculations~\ref{calc:twists} and~\ref{calc:alpha-twists} in the other cases, we see that $H^{k+i}(\Lambda^i[\bfrak^2](\lambda))=0$ for any $k>0$ and any $i \ge 0$, and therefore $H^k(\Sym[(\gf/\bfrak)^2](\lambda))=0$ for all $k>0$.
\end{proof}

\begin{remark}
    Lemma~\ref{lem:Zcohom} appears as Corollary~4.3 in~\cite{Ngo18} (and, for $\lambda = 0$ in characteristic 0, in~\cite{VilonenXue}, with a claim (Remark~6.2) that the result holds for $l\ge 5$). We have given an alternative proof, as there appears to be a small issue in~\cite{Ngo18} (and we also wish to obtain more refined information, such as about the $H^0$). The proof of Theorem~4.2 in~\cite{Ngo18} requires the case $s = 0$ (in the notation of that paper) as a inductive hypothesis, whereas it is deduced afterwards as a corollary. To repair that issue, we must show that $H^i(\Sym^n(\nf^{*,r})(\lambda))$ for all $i \ge 1$ and $\lambda \in X(T)_+$. Now, in the notation of that paper, $\beta + \lambda \in \mathcal{A}_\alpha$. Then we have the exact sequence 
    \[0 \to \Sym^{n-1}[\nf^{*,r}](\beta + \lambda) \to \Sym^n(\nf^{*,r})(\lambda) \to \Sym^{n}[\nf^{*, r-1} \oplus \nf_\beta^*](\lambda) \to 0\]
    and we can apply an appropriate inductive hypothesis to deduce the vanishing of the $i$th cohomology of the middle term from the vanishing of the $i$th cohomology of the outer terms.

    Along with brevity, the advantage of this approach is that it applies without restriction on the characteristic, as it relies on the ``Easy Lemma'' of~\cite{demazureVerySimpleProof1976} and a vanishing result of \cite{kumarFrobeniusSplittingCotangent1999} proved using Frobenius splitting.
\end{remark}

\begin{lemma}\label{lem:omegaZIY}
Suppose that $l \ge 5$. Then we have
\[H^i(Z, \omega_Z) = 0\]
for all $i > 0$ and, for $\lambda \in X(T)_+$,  
\[H^i(Z, \Ic_Y(\lambda)) = 0\]
for all $i > 0$.
\end{lemma}
\begin{proof}
From Proposition \ref{prop:sheafcalcs} that $\Ic_Y\cong\Oc_Z(\rho)$ and $\omega_Z\cong \Oc_Z(2\rho)$. The lemma now follows from Lemma~\ref{lem:Zcohom}.
\end{proof}

\begin{corollary}{\label{propYcohom}}
    Let $l \ge 5$ and $\lambda \in X(T)_+$. Then 
    \[H^i(Y, \Oc_Y(\lambda)) = 0\]
    for all $i > 0$. Furthermore, $H^i(Y, \omega_Y) = 0$ for all $i > 0$.
\end{corollary}
\begin{proof}
The first claim follows from Lemmas~\ref{lem:Zcohom} and~\ref{lem:omegaZIY} and the long exact sequence in cohomology arising from
\[0\to \Ic_Y(\lambda)\to\Oc_Z(\lambda)\to \Oc_Y(\lambda)\to 0.\]
The second claim follows from the isomorphism $\omega_Y \cong \Oc_Y(\rho)$ proved in Proposition~\ref{prop:sheafcalcs}.
\end{proof}

\section{Equations for $\Xc_{\St}$}
\label{sec:equations}

In this section, we seek explicit equations for $\Xc_{\St}$ (or equivalently, for $\Xc$). For the generic fibre, $\Xc \times_\Oc E$, this was done by the first author.

\begin{proposition}\label{prop:eqns-generic} For $n \ge 1$, $\Xc_E \subset GL_{n,E} \times GL_{n,E}$ is the closed subscheme of pairs $(\Sigma, \Phi)$ such that 
\begin{align*}\Phi \Sigma \Phi^{-1} &= \Sigma^q \\
\chi_\Phi(x) &= (x - 1)(x - q)\ldots(x - q^{n-1}) \\
\chi_{\Sigma}(x) &= (x - 1)^n.
\end{align*}
\end{proposition}
\begin{proof}
After a change of variables from the unipotent cone to the nilpotent cone via the logarithm map, this is a special case of Corollary~3.3 in \cite{funck2023geometry}. (Note that in the nilpotent version, the equations making $\log(\Sigma)$ nilpotent are contained in the ideal generated by the coefficients of $\Phi \log(\Sigma)\Phi^{-1}-q\log(\Sigma)$ --- see Proposition~2.5 of \cite{funck2023geometry}). The main idea is that, if $\Cc\subseteq \GL_{n,E}$ is the conjugacy class  of $\diag(1, q, \ldots, q^{n-1})$, which is a smooth variety defined by the second equation above, then $\Xc_E \to \Cc$ can be shown to be a vector bundle with fibre over $\Phi\in\Cc$ given by 
\[\{\Sigma \text{ unipotent} : \Phi \log(\Sigma) \Phi^{-1} = q\log(\Sigma)\}.\qedhere\]
\end{proof}

However, even for $n = 2$, the subscheme of $GL_{n, \Oc} \times GL_{n, \Oc}$ cut out by these equations is not $\Oc$-flat, and so does not coincide with $\Xc$. To find the correct set of equations, we use the method of \cite{snowdenSingularitiesOrdinaryDeformation2018}.
We begin with a slight generalisation of Proposition~2.1.5 of~\cite{snowdenSingularitiesOrdinaryDeformation2018}.

\begin{proposition}\label{prop:snowden2.1.5}
Let $F$ be a scheme of finite type over $\F$ and 
\[0\rightarrow \xi\rightarrow \epsilon \rightarrow \nu \rightarrow 0\] be a short exact sequence of vector bundles on
$F$, with $\epsilon = V \otimes_{\F} \Oc_F$ for some $\F$-vector space $V$. Let $S=\Sym(V^*)$, let $Z$ be the total space
of $\nu^*$ with $\pi : Z \to F$ the natural map, and let \[R = H^0(Z, \Oc_Z) = H^0(F,\Sym(\nu)).\] There is a natural
map of graded rings $S \to R$.

Let $\mathcal{L}$ be a line bundle on $F$ and let
\[M=H^0(Z,\pi^*\mathcal{L})=H^0(F,\Sym(\nu)\otimes \Lc),\] a graded $R$-module. Suppose that $H^i(F,\Sym(\nu)\otimes
\Lc)=0$ whenever $i>0$. Then 
    \[\Tor^S_n(M,\F)=\bigoplus_{i \ge n}H^{i-n}(F,\Lambda^i(\xi) \otimes \Lc)[i]\]
as graded vector spaces, where $W[i]$ is the vector space $W$ living in degree $i$.
\end{proposition}
\begin{proof}
    Recall from the proof of Proposition 2.1.5 of \cite{snowdenSingularitiesOrdinaryDeformation2018} that we have a
    diagram 
\[\begin{tikzcd}
	F && {F\times V} \\
	\\
	\Spec \F && V
	\arrow["{i'}", from=1-1, to=1-3]
	\arrow["{p'}"', from=1-1, to=3-1]
	\arrow["p", from=1-3, to=3-3]
	\arrow["i"', from=3-1, to=3-3]
\end{tikzcd}\] with the horizontal maps being the zero sections. Since $Z \subset F \times V$ is a closed subscheme, we
may also regard quasicoherent $\Oc_Z$-modules as quasicoherent $\Oc_{F \times V}$-modules supported on $Z$. In
particular, this applies to $\pi^*\Lc$, and if $q : F\times V \to F$ is the projection, then $\pi^*\Lc = \Oc_Z \otimes
q^*\Lc$.

There are isomorphisms
\[\Tor_{\bullet}^S(p_*\pi^*\Lc,\F)=Li^*Rp_*\pi^*\Lc=Rp'_*L(i')^*\pi^*\Lc.\] The first isomorphism results from the
hypothesis that \[R^{i}p_*\pi^*\Lc=H^i(F,\Sym(\nu)\otimes\Lc)=0\] for all $i>0$. The second is the base-change isomorphism of
\cite[\href{https://stacks.math.columbia.edu/tag/08IB}{Lemma 08IB}]{stacks-project}. 

As in the proof of Proposition 2.1.5 of \cite{snowdenSingularitiesOrdinaryDeformation2018},  the Koszul complex gives a
quasi-isomorphism of complexes of coherent sheaves on $F \times V$, graded as $\Oc_V$-modules,
\[\left[\Lambda^{i}(q^*\xi)(-i)\otimes q^*\Lc\right]\rightarrow \Oc_Z \otimes q^*\Lc = \pi^*\Lc.\] Here $(-i)$ is the shift operator on graded $\Oc_V$-modules.  The
differentials in the left-hand complex vanish on applying $(i')^*$, while the terms are flat $\Oc_{F \times V}$-modules,
and so we get a quasi-isomorphism
\[\left[\Lambda^{i}(\eta)(-i)\otimes \Lc\right]\rightarrow L(i')^*\pi^*\Lc\]
in which the differentials in the left-hand complex are zero.
Applying $Rp'_*$ we therefore find that
\[\Tor^S_n(p_*\pi^*\Lc,\F) \cong \bigoplus_{i \ge n} H^{i-n}(F, \Lambda^{i}(\eta)\otimes \Lc)[i]\]
as required.
\end{proof}

We will always apply the lemma with $V = (\gf^*)^2 = \gf^2$ and $\xi = (\nf^\perp)^2 =  \bfrak^2$, so that $\nu =
(\gf/\bfrak)^2 = (\nf^*)^2$ and $Z$ is the total space of the vector bundle $\nf^2$. Let $R = H^0(Z, \Oc_Z)$. We have
shown already that, if $l> n$, then $S \to R$ is surjective, and so $R$ is the coordinate ring of the
image $p(Z)$ of $Z$ in $\gf^2$ (a variety over $\FF$). Let $I = \ker(S \to R)$, which is also the ideal defining $p(Z)$.
Our strategy will be to write down an ideal $\tilde{I} \subset I$ and deduce that $\tilde{I} = I$ by computing 
\[I/S_+I = \Tor_1^S(R,\FF)\] using Proposition~\ref{prop:snowden2.1.5}. In the case $n = 3$, having determined $p(Z)$,
we will then determine equations for $X = p(Y)$.

Note that $S = \Sym((\gf^*)^2)$ is the coordinate ring of the space of pairs of elements of $\gf = \spl_3$, and we write
$(M, N)$ for the universal pair of matrices over $S$.

\subsection{The case $n=2$.}
In this case, $Z = Y$ is the variety whose $\F$-points are pairs of commuting elements of $\mathfrak{sl}_2$. By
Proposition~\ref{prop:snowden2.1.5} we obtain
\[I/S_+I = \Tor^S_1(R,\F)\cong H^0(\bfrak^2)[1]\oplus H^1(\Lambda^2\left(\bfrak^2\right))[2].\] Since $\bfrak \cong
 \Oc(-1)^2$, we see that 
 \[I/S_+ I\cong H^1(\Oc(-2)^6) \cong \FF^6\] is $6$-dimensional. 
 
Let $\tilde{I}$ be the ideal of $S$ generated by the entries of $MN - NM$, by $\tr(MN)$, and by $\det(M)$ and $\det(N)$.
Then $\tilde{I} \subset I$, and $\tilde{I}/S_+\tilde{I}$ is six-dimensional (as may easily be checked, noting that
$\tr(MN-NM) = 0$), sitting in degree $2$. It follows that $\tilde{I} = I$ and we have proved
\begin{theorem}\label{thm:gl2-eqns}
Let $n = 2$. Then $X \subset \gf^2$ is the closed subscheme cut out by the equations $\det(M) = \det(N) = 0$, the
entries of $MN - NM$, and $\tr(MN)$.
\end{theorem}
We can apply this to $\Xc$, recalling that we have shown that $\Xc \otimes_\Oc \F = X$.
\begin{corollary}\label{cor:Xc-eqns-sl2} Let $n = 2$. Then $\Xc \subset GL_2 \times GL_2$ is the closed subscheme of pairs $(\Sigma, \Phi)$ such that
    \begin{align*} \Phi \Sigma \Phi^{-1} &= \Sigma^q \\
        \chi_\Phi(x) &= (x-1)(x-q) \\
        \chi_\Sigma(x) &= (x-1)^2 \\
        \tr(\Phi(\Sigma - I)) &= 0.
    \end{align*}
Moreover, $\Xc_{\St}$ is defined by the same equations but with the second equation replaced by $q\tr(\Phi)^2 = 4\det(\Phi)$.
\end{corollary}
\begin{proof}
    The statement for $\Xc_{\St}$ follows from that for $\Xc$ and the discussion following Assumption~\ref{ass:nlq}. 

    To show that $\Xc$ is cut out by the given equations in $GL_2 \times GL_2$, we need only check this after
    $\otimes_{\Oc} \FF$ and $\otimes_\Oc E$. After $\otimes \FF$, and writing $\Phi = I + M$, $\Sigma = I + N$, we
    obtain exactly the equations generating the ideal $\tilde{I}$ above; thus we are done by Theorem~\ref{thm:gl2-eqns}.

    By Proposition~\ref{prop:eqns-generic}, the first three
    equations alone already define $\Xc\otimes_\Oc E$, and that this is a smooth variety. Since every point of
    $\Xc(\bar{E})$ has $\Phi \sim \diag(q,1)$, and the fourth equation is easily seen to hold at such points, we have
    that this equation is also contained in the defining ideal of $\Xc \otimes_\Oc E$.  
\end{proof}

\subsection{The case $n = 3$.}

Recall that we have $Y \subset Z$ where $Z$ is the total space of the vector bundle $\nf^2$ on $F$ and $Y$ is the
commuting subvariety, and we have $X = \Spec \Gamma(Y, \Oc_Y)$. If $R = \Gamma(Z, \Oc_Z)$ then by the discussion above we have a surjective morphism $S \to R$ with kernel $I$. 

\begin{theorem}\label{thm:Z-ideal-sl3} Suppose that $n = 3$ and that $l \ge 5$. The homogeneous ideal
    $I\trianglelefteq S$ is generated by:
\begin{itemize}
    \item $\tr(M^2) , \tr(N^2) , \tr(MN)$ (in degree 2);
    \item $\tr(M^3), \tr(N^3)$, and all the entries of
    \[M^2N, N^2M, NM^2, MN^2\]
    (in degree 3).
\end{itemize}
\end{theorem}
\begin{remark}
    In fact, the entries of $MNM$ and of $NMN$ are also in the ideal generated by these equations, which may be easily
    checked by computer. One can take any two of $M^2N, NM^2,$ and $MNM$ in the generating set.
\end{remark}

\begin{proof} By Proposition~\ref{prop:snowden2.1.5}, along with Calculation~\ref{calc3} and~\ref{calc:wedge4}, we see that:
\begin{align*} 
    I/S^+I &\cong \bigoplus_i H^{i-1}(F,\Lambda^i[\bfrak\oplus \bfrak])[i] \\
    & =\F^3[2]\oplus (\gf^4\oplus \F^{4})[3].
\end{align*}
    Hence we see that,  
        \[\dim(I/S_+I)_k = \begin{cases*}
            3 & if $k = 2$ \\
            36 & if $k = 3$ \\
            0 & otherwise. \\
        \end{cases*}\] 
        
        Let $\tilde{I}$ be the ideal of $S$ generated by the elements listed in the statement of the theorem. Since all
    the elements of $\tilde{I}$ vanish for $M, N \in \nf$, $\tilde{I}\subset I$. We show that they are equal.
    
    As $\tr(M^2), \tr(MN)$ and $\tr(N^2)$ are all linearly independent, we see that they span the degree 2 subspace of
    $\tilde{I}/S_+\tilde{I}$. By a computer calculation we see that the entries of $M^2N$ and $NM^2$ span a
    17-dimensional space modulo the subspace spanned by entries of $\tr(M)^2N$, $M\tr(MN)$, and $N \tr(M^2)$ so long as
    $l > 2$ (this is just a matter of computing the invariant factors of an explicit finitely-generated $\Z$-module).
    This same calculation also shows that the entries of $MNM$ are in $\tilde{I}$. Therefore, using that there is a
    natural bigrading on $S$ with respect to which $\tilde{I}$ is bihomogeneous, we have
    \[\dim(\tilde{I}/S^+\tilde{I})_3 = 1 + 1 + 17 + 17 = 36 = \dim(I/S^+I)_3.\] Thus $\dim(\tilde{I}/S_+\tilde{I})_k =
    \dim(I/S_+I)_k$ for all $k \ge 0$, and so $I = \tilde{I}$.
\end{proof}

\begin{theorem}\label{thm:X-ideal-sl3}
    Let $J$ be the ideal of $S$ determining the closed subscheme $X$. Then $J$ is generated by 
    \begin{itemize}
        \item $\tr(M),\tr(N)$ (in degree 1);
        \item $\tr(M^2),\tr(MN), \tr(N^2)$, and the entries of $MN-NM$ (in degree 2);
        \item and $\tr(M^3),\tr(N^3)$ and the entries of $M^2N$ and $MN^2$ (in degree 3).
    \end{itemize}
\end{theorem}
\begin{proof}
    Let $i : Y \to Z$ be the inclusion. Then we have that $H^0(X, \Oc_X) = H^0(Z, i_*\Oc_Y)$ and a short exact sequence
    \[0 \to H^0(Z,\Ic_Y) \to H^0(Z, \Oc_Z) \to H^0(Z, i_*\Oc_Y) \to 0\]
    since, by Lemmas~\ref{lem:ZYproperties} and~\ref{lem:Zcohom}, $H^1(Z, \Ic_Y)$ vanishes.

    We claim that $I_X = H^0(Z, \Ic_Y)$ is generated, as a (homogeneous) ideal of $R = H^0(Z, \Oc_Z)$, by the entries of $MN
    - NM$. These entries give eight linearly independent elements of $I_X$ and it suffices to show that 
    \[\dim(I_X\otimes_S \F)_i = \begin{cases*}
        8 & if $i = 2$ \\
        0 & otherwise.
    \end{cases*} \]

    By Lemma~\ref{lem:ZYproperties}, $\Ic_Y \cong \Oc_Z(\rho)$, locally (on $F$) generated in degree 2; hence $I_X \cong
    H^0(Z, \Oc_Z(\rho))(-2)$ as a graded $S$-module, where $(-2)$ is the shift operator. By
    Proposition~\ref{prop:snowden2.1.5}, 
    \[(I_X \otimes_S \F)_i = \dim H^{i}(F, \Oc(\rho) \otimes \Lambda^i[\bfrak \oplus \bfrak])[i+2]\] for $i \ge 0$. The claim
    then follows from Calculation~\ref{calc:twists}.

    As $J$ is the preimage of $I_X$ under $S \to R$, Theorem~\ref{thm:X-ideal-sl3} then follows from Theorem~\ref{thm:Z-ideal-sl3}.
\end{proof}

We can also write down equations for $\Xc$ as follows:

\begin{corollary}
\label{cor:Xc-equations-sl3}
    A complete set of equations defining the closed subscheme $\Xc\subseteq \GL_3\times GL_3$ over $\Oc$ is as follows:
    \begin{align*}
        \Phi \Sigma &=\Sigma^q\Phi \\
        \chi_{\Phi}(x)&=(x-q^2)(x-q)(x-1) \\
        \chi_\Sigma(x)&=(x-1)^3\\
        \tr(\Phi (\Sigma - I))&=0\\
        (\Phi-q^2)(\Sigma - I)^2&=0\\
        (\Phi-q^2)(\Phi-q)(\Sigma - I)&=0,    
    \end{align*}
    where $\chi_A(x)$ denotes the characteristic polynomial of $A$. 
    
    Similarly for $\Xc_{\St}$, where the second equation is replaced by \[q^3(\tr \Phi(x))^3 = (1 + q + q^2)^3 \det
    \Phi(x)\] and 
    \[(1 + q + q^2)^2\tr(\Phi(x)^2) = (1 + q^2 + q^4)\tr(\Phi(x))^2.\]
\end{corollary}

\begin{proof}
    The statement for $\Xc_{\St}$ follows from that for $\Xc$ and the discussion following Assumption~\ref{ass:nlq}. 

    To show that $\Xc$ is cut out by the given equations in $GL_3 \times GL_3$, we need only check this after
    $\otimes_{\Oc} \FF$ and $\otimes_\Oc E$. After $\otimes \FF$, and writing $\Phi = I + M$, $\Sigma = I + N$, we
    obtain exactly the equations in Theorem~\ref{thm:X-ideal-sl3}, as required.

    By Proposition~\ref{prop:eqns-generic}, the first three
    equations alone already generate $\Xc\otimes_\Oc E$, and that this is a smooth variety. Since every point of
    $\Xc(\bar{E})$ has $\Phi \sim \diag(q^2,q,1)$, and the fourth, fifth, and sixth equations are easily seen to hold at
    such points, we have that these equations are also contained in the defining ideal of $\Xc \otimes_\Oc E$.
\end{proof}

\section{The Weil divisor class group}
\label{sec:class-group}

We compute the Weil divisor class group of $X$ and the class of the canonical divisor. For $n = 2$ this gives
another perspective on the calculations of \cite{manning}. One might hope for similar automorphic applications for $n=3$,
but there appear to be issues caused by the failure of the natural pairing on spaces of automorphic forms to be Hecke-equivariant (we thank Jeff
Manning for explaining this point to us).

\subsection{Generalities on the class group}

Let $X$ be a Noetherian, integral, normal scheme. Recall from \cite[\href{https://stacks.math.columbia.edu/tag/0AVT}{Section 0AVT}]{stacks-project} that a coherent sheaf $\Fc$ on $X$ is \emph{reflexive} if the canonical map $\Fc \to (\Fc^\vee)^\vee$ is an isomorphism, where $\Fc^\vee = \Homf_X(\Fc,\Oc_X)$.
 A reflexive sheaf on $X$ of generic rank one is called \emph{divisorial}. The set of isomorphism classes of divisorial sheaves on $X$ forms a group that, by \cite[\href{https://stacks.math.columbia.edu/tag/0EBK}{Section 0EBK}]{stacks-project}, is isomorphic to the Weil divisor class group 
$\Cl(X)$. If $D$ is a Weil divisor on $X$ then we write $\Oc(D)$ for the associated divisorial sheaf.

If $j : U \into X$ is the inclusion of an open subscheme such that $X \setminus U$ has codimension 2, then $j^*$ and $j_*$ define quasi-inverse equivalences
of categories between the set of reflexive sheaves on $U$ and on $X$, so \[\Cl(U) \cong \Cl(X).\]
If $X$ is, in addition, Cohen--Macaulay, then the canonical sheaf $\omega_X$ is divisorial by \cite[\href{https://stacks.math.columbia.edu/tag/0AY6}{Lemma 0AY6}, \href{https://stacks.math.columbia.edu/tag/0AWN}{Lemma 0AWN}]{stacks-project} and we have $\omega_X = j_*\omega_U$. 

%If $X$ is normal then $\Div$ is injective, and if $X$ is regular (for example) then $\Div$ is an isomorphism. 
%If $Y \subset X$ is a closed subset of codimension $\ge 2$ then $\Cl(X) \cong \Cl(X \setminus Y)$.

\begin{lemma}\label{lem:weil-class} Let $f : E\rightarrow F$ be a morphism of normal integral schemes of finite
    type over a field $k$, and suppose that $f$ is faithfully flat with integral fibres. Let $K$ be the function field
    of $F$ and let $C = E_K$ be the generic fibre of $f$. Then there is an exact sequence
    \[\Gamma(C,\Oc_C^{\times})/K^\times\xrightarrow{\delta} \Cl(F)\rightarrow \Cl(E)\rightarrow \Cl(C)\rightarrow 0.\]
\end{lemma}

\begin{proof} This is \cite{fossumPicardGroupsAlgebraic1973} Proposition~1.1, noting that a normal integral scheme of
finite type over a field is Krull.  

The maps $\Cl(F) \to \Cl(E)$ and $\Cl(E) \to \Cl(C)$ are the functorial maps (see
\cite{fossumPicardGroupsAlgebraic1973}). The map $\delta$ in the exact sequence is constructed as follows. Let $u \in
\Gamma(C, \Oc_C^\times)$. Then $u$ determines an element of the function field of $E$ and hence a divisor $\Div(u)$ on
$E$, which one can show is the pullback of a divisor $D$ on $F$. We define $\delta(u) = \Oc(D) \in \Cl(X)$.
\end{proof}

\begin{lemma}\label{lem:linebunpic} Suppose that $F$ is as in Lemma~\ref{lem:weil-class}, that $f : \Lc \to
F$ is a line bundle, and that $E$ is the complement of the zero-section of $f$. 

Choose an isomorphism $\Lc_K \to \Spec K[t]$, so that $t$ is a generator of \[\Gamma(E_K, \Oc_{E_K}^\times)/K^\times.\] Then \[\delta(t) = \Lc^{-1}\] where $\delta$
is the map from Lemma~\ref{lem:weil-class}. 
\end{lemma}
\begin{proof}
    If $D\subset F$ is a closed integral subscheme of codimension 1 with generic point $d$, and if we choose a trivialisation $\Lc_d \cong \Spec \Oc_{F,d}[s]$ such that $s = 0$ is the zero section
    then we have 
    \[t = \kappa_D s\]
    for some $\kappa_D \in K$. We have that $v_d(\kappa_D) = 0$ for all but finitely many $D$ and, by definition, 
    \[\delta(t) = \Oc\left(\sum_{D}v_d(\kappa_D)D\right).\]

    On the other hand, regard $\Lc$ as a sheaf on $X$ that is locally free of rank 1, and let $t^\vee \in \Gamma(\Spec(K), \Lc_K)$ correspond to the section $t^\vee : \Spec(K) \to \Spec K[t]$ such that $(t^\vee)^*(t) = 1$. Similarly, we have $s^\vee \in \Gamma(X_{d}, \Lc_d)$, and the relation $t^\vee = \kappa_D^{-1} s^\vee$ for each $D$ as above. Then we have 
    \[\Lc = \Oc\left(\sum_D v_d(\kappa_D^{-1})D\right),\]
    from which the result follows.
\end{proof}

\subsection{Calculation of $\Cl(X)$}

Recall that we have the diagram 
\[\begin{tikzcd}
 Y \arrow[d,"\pi"] \arrow[r,"f"] & X  \\
 F \end{tikzcd}\] where $Y$ is an irreducible variety, the map $f$ is projective and $\pi$ is a fibre bundle with fibres
 isomorphic to 
 \[C(\nf) = \{M, N \in \nf : [M, N] = 0\},\] by (the proof of) Lemma~\ref{lem:ZYproperties} part~(2). Let $U\subset X$
be the open subscheme of points $(M,N)\in X$ such that either $M$ or $N$ is regular nilpotent. If $V = f^{-1}(U)$ then
$f|_V : V \to U$ is an isomorphism, as in Lemma~\ref{lem:birational}, while $\pi|_V : V \to F$ is a fibre bundle with
fibres \[C(\nf)^{\reg} = \{M, N \in C(\nf) : \text{$M$ or $N$ regular}\}.\]
\begin{lemma}\label{lem:YtoF} \begin{enumerate}
    \item When $n = 2$ 
    \[C(\nf)^\reg\cong\AA^2\setminus\{(0,0)\},\] $Y \setminus V$ has codimension 2 in $Y$, and $X \setminus U$ has
    codimension 3 in $X$.
    \item When $n = 3$ 
    \[C(\nf)^\reg\cong\AA^2\setminus\{(0,0)\}\times \GG_m\times \AA^2,\]
    $Y \setminus V$ has codimension 1 in $Y$, and $X \setminus U$ has codimension 2 in $X$.
\end{enumerate}
\end{lemma}

\begin{proof}
If $n = 2$, then $C(\nf) = \nf \times \nf \cong \AA^2$ and $C(\nf)^{\reg} = \AA^2 \setminus \{(0,0)\}$. Therefore
\[\codim(Y \setminus V \subset Y) = \codim(\{0,0\} \subset \AA^2) = 2.\] Now \[U = \pi(V) = \{(0,0)\} \in X \subset \gf
\times \gf\] and $X$ has dimension $3$, so $\codim(X\setminus U \subset X) = 3$. This proves part~(1).

Now suppose that $n = 3$. We write
\[C(\nf)=\left\{\left(\begin{pmatrix}
0 & a & b \\
0 & 0 & c \\
0 & 0 & 0
\end{pmatrix},\begin{pmatrix}
0 & d & e \\
0 & 0 & f \\
0 & 0 & 0
\end{pmatrix}
\right):af=cd\right\},\]
which has dimension 5, and see that 
\begin{align*}C(\nf) \setminus C(\nf)^{\reg} &= \VV(af-cd, ac,df) \\
&= \VV(a,d)\cup \VV(c,f).\end{align*}
Therefore \[\codim(Y \setminus V \subset Y) = \codim(C(\nf)\setminus C(\nf)^{\reg} \subset C(\nf)) = 1.\]
We have an isomorphism 
\begin{align*}
    \AA^2\setminus \{(0,0)\}\times \GG_m\times \AA^2&\rightarrow C(\nf)^{\reg} \\
    (a,d,\lambda,b, e)&\mapsto (a,b,c = \lambda a, d,e,f = \lambda d)
\end{align*}
with inverse defined by $\frac{c}{a}\mapsto \lambda$ when $a\neq 0$, and $\frac{f}{d}\mapsto \lambda$ when $d\neq 0$.

Finally, we compute the dimension of $X \setminus U$. Since the locus $M = N  = 0$
has codimension 8 in $X$, it is enough to compute the dimension of $\{(M, N) \in X \setminus U : M \ne 0\}$. The
projection map from this subset to $\{M \in \gf : M^2 = 0, M \ne 0\}$ is a $GL_3$-equivariant surjection. Its image is
the orbit in $\gf$ of the matrix $M_0 = \begin{pmatrix} 0 & 0 & 1 \\ 0 & 0 & 0 \\ 0 & 0 & 0 \end{pmatrix}$ whose centraliser
consists of all matrices of the form
\[\begin{pmatrix}
     a & b & c \\
     0 & d & e \\
     0 & 0 & a
\end{pmatrix}\]
and whose orbit in $\gf$ has dimension $4$ (by the orbit-stabiliser theorem). The fibre in $X\setminus U$ of $M_0$ is then 
\[\left\lbrace\begin{pmatrix} 0 & b & c \\
     0 & 0 & e \\
     0 & 0 & 0 \end{pmatrix}: be=0\right\rbrace\]
which has dimension $2$, and so \[\dim(X \setminus U) = 4 + 2 = 6 = \dim(X) - 2\]
as required.
\end{proof}

Part~(1) of the next theorem is \cite{manning} Proposition~3.14 part~(1), proved there using toric geometry.

\begin{theorem}
\begin{enumerate}
\item Suppose that $n = 2$. There is an isomorphism \[\nu : X^*(T) \to \Cl(X).\]
Hence $\Cl(X) \cong \Z$.

\item Suppose that $n = 3$. There is a surjective homomorphism $\nu : X^*(T) \to \Cl(X)$ with kernel generated by $3(L_1 + L_3)$.
Therefore
\[\Cl(X)\cong X^*(T)/\pres{3(L_1 + L_3)} \cong \Z \times \Z/3\Z.\]
\end{enumerate}
\end{theorem}
\begin{proof}
The case $n = 2$ follows quickly from Lemma~\ref{lem:YtoF} part~(1): we have 
\[\Cl(X) = \Cl(U) = \Cl(V) = \Cl(Y) = \Cl(F)\]
since $X \setminus U$ has codimension 2 in $X$, $Y \setminus V$ has codimension 3 in $Y$, and $Y$ is a vector bundle over $F$ and so $\Cl(Y) \cong \Cl(F)$ (for example, by Lemma~\ref{lem:weil-class}). As is well-known, \[\Cl(F) = \Pic(F) = X^*(T) \cong \Z.\]

Suppose now that $n = 3$. By Lemma~\ref{lem:YtoF} part~(2), as the codimension of $X\setminus U$ in $X$ is 2, we have \[\Cl(X) = \Cl(U) = \Cl(V).\] Let $K$ be the function field of $F$. By Lemma~\ref{lem:weil-class} and Lemma~\ref{lem:YtoF} part~(2), we have an exact sequence
\[\Gamma(C(\nf)^\reg_K, \Oc_{C(\nf)^{\reg}_K}) \xrightarrow{\delta} \Cl(F) \to \Cl(V) \to \Cl(C(\nf)^\reg_K) \to 0.\]
Since $C(\nf)^\reg_K$ is an open subscheme of $\AA^5_K$, $\Cl(C(\nf)^\reg_K) = 0$. We have \[\Cl(F) = \Pic(F) = X^*(T) \cong \Z^2.\]
It remains to compute the image of $\delta$.

By Lemma~\ref{lem:YtoF}, we have a morphism $h : C(\nf)_K \to \Spec K[\lambda, \lambda^{-1}]$ that induces an isomorphism
\[\Gamma(C(\nf)^\reg_K, \Oc_{C(\nf)^{\reg}_K}) \isomto  \{\lambda^n : n \in \Z\}.\]
This morphism is $B$-equivariant if $\lambda = \frac{c}{a} = \frac{f}{d}$ is given weight $L_1 - L_2 - (L_2 - L_3) = 3(L_1 + L_3)$, and so we obtain a morphism 
\[V \to \Oc(-3(L_1 + L_3)) \setminus \{\text{zero section}\}\]
whose restriction to the generic fibres is $h$. We may therefore apply Lemma~\ref{lem:linebunpic} and deduce that 
\[\delta(\lambda) = \Oc(3(L_1 + L_3)).\]
The result follows.
\end{proof}
Finally, we compute the canonical sheaf of $X$ in terms of these isomorphisms. When $n = 2$, this recovers \cite{manning} Proposition~3.14 part~(2).
\begin{corollary}
\begin{enumerate}
    \item Suppose that $n = 2$. Then $\omega_X = \nu(2\rho)$. 
    \item Suppose that $n = 3$ and $l \ge 11$. Then $\omega_X=\nu(\rho)\in Cl(X)$.
\end{enumerate}
\end{corollary}
\begin{proof}
    In either case, by Theorem~\ref{thm:resolution}, we have that $\omega_X=f_*\omega_Y$. 

    If $n = 2$ then we have $\omega_Y = \pi^*\Oc(2\rho)$ by Proposition~\ref{prop:sheafcalcs}.
    We therefore have 
    \[\omega_X = f_*\pi^*\Oc(2\rho),\]
    but also 
    \[j^*\nu(2\rho) = f_*(\pi^*\Oc(2\rho)|_V) = j^*(f_*\pi^*\Oc(2\rho))\]
    as $f : V \to U$ is an isomorphism. Thus $\omega_X=  \nu(2\rho)$.
    
    If $n = 3$ we again apply Proposition~\ref{prop:sheafcalcs} to obtain 
    \[\omega_X = f_* \pi^*\Oc(\rho)\]
    but also 
    \[j^*\nu(\rho) = f_*(\pi^*\Oc(\rho)|_V) = j^*(f_*\pi^*\Oc(\rho))\]
    as in the case $n = 2$.
    Thus $\omega_X = \nu(\rho)$.
\end{proof}

\subsection{Speculation on patched modules and their multiplicity}
\label{sec:multiplicity-new}

Let $\iota : GL_3 \to GL_3$ be the involution $A \mapsto A^{-T}$, which induces involutions on $X$ and on $F$ that we also denote by $\iota$. On $X^*(T)$ this induces the involution interchanging $L_1$ with $-L_3$.

Let $D$ be a division algebra of degree $3$ over a local field $F$ with residue field of order $q$. One expects (as in \cite[Conjecture 4.5.1]{Zhu20}) a functor from the category of finitely generated
smooth $\FF$-representations of $D^\times$ to the derived category of coherent sheaves on $X$. In a global context, the completion of this functor at the origin of $X$ should be realised by a ``patching functor'' constructed using an appropriate unitary group. Applying this functor to $\ind_{F^\times\Oc_D^\times}^{D^\times}(\FF)$ we expect to obtain a coherent sheaf $\Mc$ on $X^\wedge_0$ with the following properties:
\begin{enumerate}
    \item $\Mc$ is maximal Cohen--Macaulay on $X^\wedge_0$ of generic rank one (and hence divisorial);
    \item $\iota$-self-duality: $\Homf_{X^\wedge_0}(\Mc, \omega_{X^\wedge_0}) \cong \iota^*\Mc$.
\end{enumerate}

In \cite{manning}, it was shown (in the patching context) that these properties characterise $\Mc$ in the case $n = 2$, and this was used to prove a new global multiplicity 2 result for mod $l$ automorphic forms. For that it is necessary to pass to the completion of $X$, which is subtle in the context of the class group. As there are other, more serious, issues arising from the natural pairing on the patched module not being Hecke-equivariant, we don't pursue this and simply ask (for $n = 3$) to what extent the above properties characterise $\Mc$ on $X$.

\begin{proposition}\label{prop:self-dual-possiblities}
    Suppose that $n = 3$ and that $\Mc$ is a divisorial sheaf on $X$ with $\Mc = \nu(w)$ for $w\in X(T)$. Then $\Homf_X(\Mc, \omega_X) \cong \iota^*\Mc$ if and only if 
    \[w = L_1, -L_3, \text{or } 2L_1 + L_3.\]
\end{proposition}
\begin{proof}
    The map $\nu$ is easily seen to be $\iota$-equivariant. If $\Mc = \nu(aL_1 - bL_3)$ we have
    \[\Homf_X(\Mc, \omega_X) = \nu(L_1 - L_3 - (aL_1 - bL_3)\]
    and
    \[\iota^*\Mc = \nu(bL_1 - aL_3).\]
    These are equal if and only if 
    \[L_1 - L_3 = (a + b)(L_1 - L_3) + 3k(L_1 + L_3)\]
    for some $k \in \Z$. This holds if and only if $a + b = 1$ and $k = 0$. The different possibilities for $b \bmod 3$ give the result.
\end{proof}

We can explicitly describe $\nu(w)$ in these cases.

\begin{lemma}\label{lem:CM-criterion} Suppose that $w \in X^*(T)$ is such that $Rf_*\Oc(w)$ and $Rf_*(\rho - w)$ are each concentrated in degree zero. Then 
$f_*\Oc(w)$ is Cohen--Macaulay and
\[\nu(w) = f_*\Oc(w).\]
\end{lemma}

\begin{proof} That $f_*\Oc(w)$ is Cohen--Macaulay is \cite{haconRationalityKawamataLog2019} Proposition~2.2. It is therefore reflexive. The final claim follows as, if $j : U \into X$ is the natural inclusion, then $X \setminus U$ has codimension 2 and $j^*f_*\Oc(w) = j^*\nu(w)$ by construction. 
\end{proof}

\begin{proposition} Let $l \ge 5$ and $\lambda = L_1, -L_3$, or $2L_1 + L_3$ (as in Proposition~\ref{prop:self-dual-possiblities}). Then 
\[\nu(\lambda) = f_*\Oc(\lambda)\]
is Cohen--Macaulay.
\end{proposition}

\begin{proof}
    By Proposition~\ref{propYcohom}, $H^i(Y, \Oc_Y(\lambda))$ and $H^i(Y, \Oc_Y(\rho - \lambda))$ vanish 
    for each of these values of $\lambda$ and for all $i > 0$. Since $X$ is affine, we obtain exactly the vanishing of $R^if_*$ required to apply Lemma~\ref{lem:CM-criterion} to these $\lambda$,
    and the Proposition follows. 
\end{proof}

Let $x \in X$ be the point $(0,0)$. If $\Mc$ is a coherent $\Oc_X$-module, then we call $\dim_\F(\Mc\otimes_{\Oc_x}\FF)$ the \emph{multiplicity} of $\Mc$ (at the origin). When $\Mc^\wedge_x$ arises as a patched module, this multiplicity will be the multiplicity of a system of Hecke eigenvalues in a certain space of mod $l$ automorphic forms, as in \cite{manning} (see also \cite[Section~4.1]{calegari-geraghty} when $l =p$). When $\Mc = \omega_X$, this number is known as the \emph{type} of the Cohen--Macaulay local ring $\Oc_{X,x}$.

\begin{proposition} Let $\lambda \in \{L_1, -L_3, 2L_1 + L_3, L_1 - L_3\}$ and let $l\ge 5$ with $l \ge 7$ if $\lambda = 2L_1 + L_3$. Let $m(\lambda)$ be the multiplicity of $\nu(\lambda) = f_*\Oc_Y(\lambda)$.
Then $m(\lambda)$ is as in Table~\ref{tab:multiplicities}.
\begin{table}[h]
\[
    \begin{array}{c|cccc}
         \lambda & L_1 & -L_3 & 2L_1 + L_3 & L_1 - L_3 \\
         \hline
         m(\lambda) & 3 & 3 & 16 & 8 
    \end{array}
\]
    \caption{Multiplicities of $\nu(\lambda)$.}
    \label{tab:multiplicities}
\end{table}
\end{proposition}
\begin{proof} 
Let $\lambda$ be one of the weights occurring in the proposition and let $M = H^0(Z, \Oc_Z(\lambda))$ and $N = H^0(Y, \Oc_Y(\lambda))$. As in section~\ref{sec:equations}, by Lemma~\ref{lem:Zcohom} we are in the situation of Proposition~\ref{prop:snowden2.1.5} with $F$ the flag variety of $\GL_3$, the short exact sequence 
\[0 \to \xi \to \epsilon \to \eta \to 0\]
being 
\[0\to\bfrak^2\to \gf^2\to (\gf/\bfrak)^2\to 0,\]
and $\Lc = \Oc_F(\lambda)$. We define $S=\Sym[(\gf^{*})^2]$ as in Proposition~\ref{prop:snowden2.1.5}, and define 
\[R=H^0(Z,\Oc_Z)=H^0(F,\Sym[(\gf/\bfrak)^2])\] Thus (and recalling that, as in the proof of Theorem~\ref{thm:X-ideal-sl3}, $H^1(Z, \Ic_Y)$ vanishes) we have morphisms of graded algebras 
\[S\twoheadrightarrow R\twoheadrightarrow H^0(Y,\Oc_Y)\isomto H^0(X,\Oc_X)\]
which fit into the following diagram:
\[\begin{tikzcd}
	& {\Fc\times \Spec(S)} & {\Spec(S)} \\
	\Fc & Z & {\Spec(R)} \\
	& Y & X
	\arrow["f", from=3-2, to=3-3]
	\arrow[hook, from=3-2, to=2-2]
	\arrow[hook, from=2-3, to=1-3]
	\arrow[hook, from=2-2, to=1-2]
	\arrow[from=1-2, to=1-3]
	\arrow[from=2-2, to=2-3]
	\arrow[hook, from=3-3, to=2-3]
	\arrow[from=1-2, to=2-1]
	\arrow[from=2-2, to=2-1]
	\arrow["\pi", from=3-2, to=2-1]
\end{tikzcd}\]
Here, the schemes on the right are all affine, the schemes in the middle are all fibre bundles over $F$, and the vertical arrows are closed immersions.

Let $M = H^0(Z, \Oc(\lambda))$, a graded $S$-module. Then Proposition~\ref{prop:snowden2.1.5} applies and we have 
\[M\otimes_S \FF = \bigoplus_{i \ge 0} H^i(F, \Lambda^i(\bfrak^2)(\lambda))[i] = \begin{cases}V(\lambda)[0] & \text{if $\lambda \ne 2L_1 + L_3$} \\ (\gf \oplus \gf)[1] & \text{if $\lambda = 2L_1 + L_3$}\end{cases}\]
by Calculations~\ref{calc:twists} and~\ref{calc:alpha-twists}. In either case we see that $\dim(M\otimes_S\FF)$ is as given in Table~\ref{tab:multiplicities}.

By Proposition~\ref{propYcohom} we have 
\[H^i(Z, \Ic_Y(\lambda)) = H^i(Z, \Oc(\lambda + \rho)) = 0\] for $i > 0$. We therefore have a short exact sequence of graded $S$-modules
\[0 \to H^0(Z, \Ic_Y(\lambda)) \to H^0(Z, \Oc_Z(\lambda)) = M \to H^0(Y, \Oc_Y(\lambda))\to 0. \]
Since $\pi_*\Ic_Y(\lambda)$ is (locally on $F$) generated in degree 2, $H^0(Z, \Ic_Y(\lambda))$ is generated in degrees $\ge 2$. We have shown that $M$ is generated in degrees $\le 1$, and therefore
\[M \otimes_S \FF \to H^0(Y, \Oc_Y(\lambda))\otimes_S\FF\]
is an isomorphism. As 
\[\dim H^0(Y, \Oc_Y(\lambda))\otimes_S \FF = \dim H^0(X, \nu(\lambda)) \otimes_S\FF = m(\lambda),\]
we obtain the result.
\end{proof}

\bibliographystyle{amsalpha}
\bibliography{references.bib}
\end{document}